\documentclass[11pt]{article}
\usepackage{amssymb}
\usepackage{amsmath,amsthm}
\usepackage[latin1]{inputenc}
\usepackage{color}
\usepackage{tikz,ifthen}
\usepackage{enumerate}
\usepackage{graphicx}
\usepackage{hyperref}
\hypersetup{colorlinks=true, linkcolor=blue, citecolor=blue, urlcolor=blue}

\newcommand{\ZZ}{\mathbb{Z}}

\newtheorem{remark}{Remark}

\newtheorem{lemma}[remark]{Lemma}
\newtheorem{theorem}[remark]{Theorem}
\newtheorem{proposition}[remark]{Proposition}
\newtheorem{corollary}[remark]{Corollary}
\newtheorem{claim}[remark]{Claim}
\newtheorem{example}[remark]{Example}

\title{Uniquely identifying the edges of a graph: the edge metric dimension}

\author{Aleksander Kelenc$^{(1)}$, Niko Tratnik$^{(1)}$ and Ismael G. Yero$^{(2)}$\\
    \\
$^{(1)}${\small Faculty of Natural Sciences and Mathematics}\\
{\small University of Maribor,}  {\small Koro\v{s}ka 160, 2000 Maribor, Slovenia.} \\
{\small niko.tratnik\@@gmail.com}, {\small Aleksander.Kelenc\@@um.si}\\
$^{(2)}${\small Departamento de Matem\'aticas, Escuela Polit\'ecnica Superior de Algeciras}\\
{\small Universidad de C\'adiz,} {\small
Av. Ram\'on Puyol s/n, 11202 Algeciras, Spain.} \\ {\small
ismael.gonzalez\@@uca.es}
}

\begin{document}

\maketitle

\begin{abstract}
Let $G=(V,E)$ be a connected graph, let $v\in V$ be a vertex and let $e=uw\in E$ be an edge. The distance between the vertex $v$ and the edge $e$ is given by $d_G(e,v)=\min\{d_G(u,v),d_G(w,v)\}$. A vertex $w\in V$ distinguishes two edges $e_1,e_2\in E$ if $d_G(w,e_1)\ne d_G(w,e_2)$. A set $S$ of vertices in a connected graph $G$ is an edge metric generator for $G$ if every two edges of $G$ are distinguished by some vertex of $S$. The smallest cardinality of an edge metric generator for $G$ is called the edge metric dimension and is denoted by $edim(G)$. In this article we introduce the concept of edge metric dimension and initiate the study of its mathematical properties. We make a comparison between the edge metric dimension and the standard metric dimension of graphs while presenting some realization results concerning the edge metric dimension and the standard metric dimension of graphs. We prove that computing the edge metric dimension of connected graphs is NP-hard and give some approximation results. Moreover, we present some bounds and closed formulae for the edge metric dimension of several classes of graphs.
\end{abstract}



\section{Introduction}

A {\em generator} of a metric space is a set $S$ of points in the space with the property that every point of the space is uniquely determined by its distances from the elements of $S$. Nowadays there exist some different kinds of metric generators in graphs, each one of them studied in theoretical and applied ways, according to its popularity or to its applications. Nevertheless, there exist quite a lot of other points of view which are still not completely taken into account while describing a graph throughout these metric generators. In this investigation we introduce and study a new style of metric generators in order to contribute to the knowledge on these distance-related parameters in graphs.

Given a simple and connected graph $G=(V,E)$, consider the metric $d_G:V\times V\rightarrow \mathbb{R}^+$, where $d_G(x,y)$ is the length of a shortest path between $x$ and $y$.  A vertex $v\in V$ is said to \emph{distinguish}\footnote{Throughout the article, we also use the terms ``recognize'' or ``determine'' instead of ``distinguish'' to describe the same property.} two vertices $x$ and $y$, if $d_G(v,x)\ne d_G(v,y)$. Also, the set $S\subset V$ is said to be a \emph{metric generator} for $G$ if any pair of vertices of $G$ is distinguished by some element of $S$. A minimum generator is called a \emph{metric basis}, and its cardinality the \emph{metric dimension} of $G$, denoted by $dim(G)$. This is the basic or standard case of metric generators in graphs and, at this moment, one of the most commonly in the literature.

This primary concept of metric dimension was introduced by Slater in \cite{leaves-trees}, where the metric generators were called \emph{locating sets}, in connection with the problem of uniquely recognizing the location of an intruder in a network. Also, the concept of metric dimension of a graph was introduced independently by Harary and Melter in \cite{harary}, where metric generators were called \emph{resolving sets}. Several applications of this invariant to the navigation of robots in networks are discussed in \cite{landmarks} and applications to chemistry in \cite{chartrand,chartrand1,pharmacy1}. Furthermore, this topic has some applications to problems of pattern recognition and image processing, some of which involve the use of hierarchical data structures \cite{Melter1984}. Some interesting connections between metric generators in graphs and the Mastermind game or coin weighing have been presented in \cite{Caceres2007}. Moreover, we refer the reader to the work \cite{Bailey2011a}, where it can be found some historical evolution, nonstandard terminologies and more references on this topic.

On the other hand, in order to discuss different points of view of metric generators, several authors have introduced other variations of metric generators. For instance, resolving dominating sets \cite{brigham}, independent resolving sets \cite{chartrand3}, local metric sets \cite{LocalMetric}, strong resolving sets \cite{Oellermann}, resolving partitions \cite{chartrand2}, strong resolving partitions \cite{GonzalezYero2013}, etc. have been presented and studied. A few other very interesting articles concerning metric dimension of graphs can be be found in the literature. However, according to the amount of results on this topic, we prefer to cite only those papers which are important from our point of view.

A metric basis $S$ of a connected graph $G$ uniquely identifies all the vertices of $G$ by mean of distance vectors. One could think that also the edges of the graph are also identified by $S$ with respect to distances to $S$. However, this is further away from reality. For instance, Figure \ref{example1}
shows an example of a graph, where no metric basis uniquely recognizes all the edges of the graph. We observe that the graph $G$ of the Figure \ref{example1} satisfies that $dim(G)=2$ and the whole set of metric bases are the following ones: $\{1,3\}$, $\{7,9\}$, $\{7,11\}$, $\{7,13\}$, $\{9,10\}$, $\{9,12\}$, $\{10,11\}$, $\{10,13\}$, $\{11,12\}$ and $\{12,13\}$. But, for each one of these metric bases, there exists at least a pair of edges which is not distinguished by the corresponding basis.

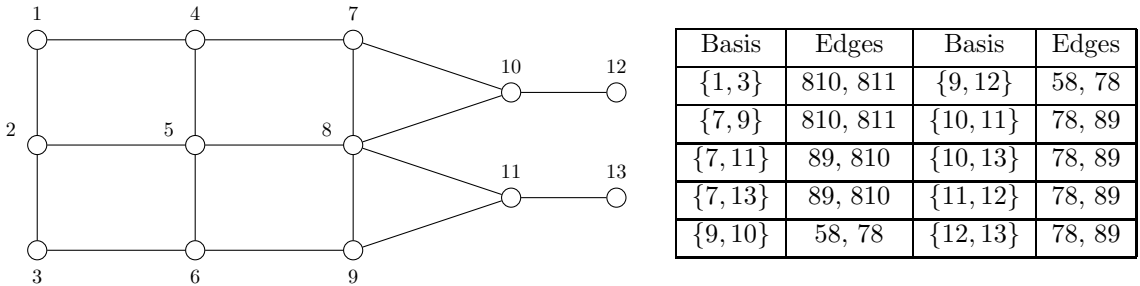
\begin{figure}[h]
\begin{minipage}[b]{.4\linewidth}
\centering
\begin{tikzpicture}[scale=.7, transform shape]
\node [draw, shape=circle] (a3) at  (0,0) {};
\node [draw, shape=circle] (a2) at  (0,2) {};
\node [draw, shape=circle] (a1) at  (0,4) {};
\node [draw, shape=circle] (a6) at  (3,0) {};
\node [draw, shape=circle] (a5) at  (3,2) {};
\node [draw, shape=circle] (a4) at  (3,4) {};
\node [draw, shape=circle] (a9) at  (6,0) {};
\node [draw, shape=circle] (a8) at  (6,2) {};
\node [draw, shape=circle] (a7) at  (6,4) {};
\node [draw, shape=circle] (a11) at  (9,1) {};
\node [draw, shape=circle] (a10) at  (9,3) {};
\node [draw, shape=circle] (a13) at  (11,1) {};
\node [draw, shape=circle] (a12) at  (11,3) {};

\draw(a1)--(a2)--(a3)--(a6)--(a5)--(a4)--(a7)--(a8)--(a9)--(a11)--(a8)--(a10)--(a7);
\draw(a2)--(a5)--(a8);
\draw(a1)--(a4);
\draw(a6)--(a9);
\draw(a10)--(a12);
\draw(a11)--(a13);

\node at (0,-0.5) {$3$ };
\node at (-0.5,2.3) {$2$ };
\node at (0,4.5) {$1$ };
\node at (3,-0.5) {$6$ };
\node at (2.5,2.3) {$5$ };
\node at (3,4.5) {$4$ };
\node at (6,-0.5) {$9$ };
\node at (5.5,2.3) {$8$ };
\node at (6,4.5) {$7$ };
\node at (9,1.5) {$11$ };
\node at (9,3.5) {$10$ };
\node at (11,1.5) {$13$ };
\node at (11,3.5) {$12$ };
\end{tikzpicture}
\par\vspace{0pt}
\end{minipage}
\begin{minipage}[b]{.75\linewidth}
\centering
\small{
\begin{tabular}{|c|c|c|c|}
  \hline
 Basis & Edges & Basis & Edges \\[2pt]  \hline
 $\{1,3\}$ & $810$, $811$ & $\{9,12\}$ & $58$, $78$ \\[2pt] \hline
 $\{7,9\}$ & $810$, $811$ & $\{10,11\}$ & $78$, $89$ \\[2pt]  \hline
 $\{7,11\}$ & $89$, $810$ & $\{10,13\}$ & $78$, $89$ \\[2pt]  \hline
 $\{7,13\}$ & $89$, $810$ & $\{11,12\}$ & $78$, $89$ \\[2pt]  \hline
 $\{9,10\}$ & $58$, $78$ & $\{12,13\}$ & $78$, $89$ \\[2pt]  \hline
\end{tabular}}
\par\vspace{12pt}
\end{minipage}
\caption{A graph where any metric basis does not recognizes all edges and a table with all metric bases and two edges which are not recognized by the corresponding metric basis.}
\label{example1}
\end{figure}

In this sense, a natural question is: Are there some sets of vertices which uniquely identify all the edges of a graph? The answer is, of course, positive, and it is our goal to study such sets in this work. That is, the present research is centered in a new variant of metric generators in graphs, which is oriented to uniquely determine the edges of a graph. Given a connected graph $G=(V,E)$, a vertex $v\in V$ and an edge $e=uw\in E$, the distance between the vertex $v$ and the edge $e$ is defined as $d_G(e,v)=\min\{d_G(u,v),d_G(w,v)\}$. A vertex $w\in V$ \emph{distinguishes} (\emph{recognizes} or \emph{determines}) two edges $e_1,e_2\in E$ if $d_G(w,e_1)\ne d_G(w,e_2)$. A set $S$ of vertices in a connected graph $G$ is an \emph{edge metric generator} for $G$ if every two edges of $G$ are distinguished by some vertex of $S$. The smallest cardinality of an edge metric generator for $G$ is called the \emph{edge metric dimension} and is denoted by $edim(G)$. An \emph{edge metric basis} for $G$ is an edge metric generator for $G$ of cardinality $edim(G)$.

Another useful approach for edge metric generators could be the following one. Given an ordered set of vertices $S=\{s_{1}, s_{2}, \ldots, s_{d}\}$ of a connected graph $G$, for any edge $e$ in $G$, we refer to the $d$-vector (ordered $d$-tuple) $r(e|S)=\left(d_{G}(e,s_{1}),d_{G}(e,s_{2}),\ldots,d_{G}(e,s_{d})\right)$ as the \emph{edge metric representation} of $e$ with respect to $S$. In this sense, $S$ is an edge metric generator for $G$ if and only if for every pair of different edges $e_1,e_2$ of $G$, it follows $r(e_1|S)\neq r(e_2|S)$.

Once defined the concept of edge metric generator, which uniquely determines every edge of the graph, one could think that probably any edge metric generator $S$ is also a standard metric generator, \emph{i.e.} every vertex of the graph is identified by $S$. Again, this is further away from the reality, even so there are several families in which such a fact occurs. We just have to take for instance the hypercube graph $Q_4$, for which is known from \cite{chartrand} that $dim(Q_4)=4$, and we have computed in this work that $edim(Q_4)=3$ (such computation was done by a computer program using an exhaustive search algorithm). According to such facts, we herewith initiate the study of edge metric generators in graph, throughout analyzing several relationships between $dim(G)$ and $edim(G)$ for several classes of graphs $G$. Moreover, we make a complexity analysis of the problem of computing the edge metric dimension of graphs. Finally, we present some bounds on the edge metric dimension of graphs.

\section{Edge metric generators and metric generators}

In this work we introduce the edge metric dimension of a graph and a first natural question concerns with the existence of graphs with predetermined values of such a new parameter. That is for instance, given two integers $r,n$ with $1\le r\le n-1$: Is there a connected graph $G$ of order $n$ such that $edim(G)=r$? The answer of such question is yes and to see this, we proceed in the following way. If $r=n-1$ or $r=1$, then we take the complete graph $K_n$ or the path graph $P_n$, respectively. On the contrary ($2\le r\le n-2$), we can easily check the positive answer by constructing a tree $T_{r,n}$ in the following way. We begin with a star graph $S_{1,r}$. Then we add a path with $n-r-1$ vertices and add an edge between a leaf of the path and the center of the star $S_{1,r}$. It is straightforward to observe that such a tree $T_{r,n}$ has order $n$ and edge metric dimension $r$ (the $r$ leaves of the star form an edge metric generator of $T_{r,n}$).

Since the metric dimension and the edge metric dimension are closely related, another realization result regarding our new parameter is clearly connected with considering them together. That is, given three integers $r,t,n$ with $1\le r,t\le n-1$: Is there a connected graph $G$ of order $n$ such that $dim(G)=r$ and $edim(G)=t$? In contrast with the first realizability question, the answer to this second question seems to be harder to answer. One reason is based on the fact that there is not specifically clear relationship between $dim(G)$ and $edim(G)$ for a graph $G$, as we have stated in the Introduction. Namely, it is possible to find graphs for which the metric dimension equals the edge metric dimension, as well as other graphs $G$ for which  $dim(G) < edim(G)$ or $edim(G) < dim(G)$. It is now our goal to explore such situations by comparing the values of $dim(G)$ and $edim(G)$ for several families of connected graphs and further focus in the realization question stated above.

\subsection{Graphs for which $dim(G)=edim(G)$}\label{dim=edim}

The equality $dim(G)=edim(G)$ is attained for several basic families of graphs. In several cases, obtaining the value of the edge metric dimension of a graph $G$ is quite similar to computing the metric dimension of $G$. In such situation we just state the result without proof. We precisely begin this section with such classes of graphs, namely paths $P_n$, cycles $C_n$ or complete graph $K_n$.

\begin{remark}\label{rem-path-1}
For any integer $n\ge 2$, $edim(P_n)=dim(P_n)=1$, $edim(C_n)=dim(C_n)=2$ and $edim(K_n)=dim(K_n)=n-1$. Moreover, $edim(G)=1$ if and only if $G$ is a path $P_n$.
\end{remark}

If $K_{r,t}$ is a complete bipartite graph different from $K_{1,1}$, then it is known that $dim(K_{r,t}) = r+t-2$. Next we show that the same is true for the edge metric dimension.

\begin{remark}
For any complete bipartite graph $K_{r,t}$ different from $K_{1,1}$, $edim(K_{r,t}) =dim(K_{r,t})$ $ = r+t-2$.
\end{remark}

\begin{proof}
Let $V$ and $U$ be the bipartition sets of $K_{r,t}$. For a first inequality, suppose that $S$ is an edge metric generator without two elements of $V$, \emph{i.e.} there are two vertices $x,y \in V$ such that $x$ and $y$ are not in $S$. Hence, let $u \in U$ and consider the edges $e=ux$ and $f=uy$. It follows that $e$ and $f$ have distance $0$ to $u$ and distance $1$ to every other element in $S$. Therefore, $S$ is not an edge metric basis, a contradiction. We similarly proceed with the set $U$ and it follows that any edge metric generator must contain all but (maybe) one element of every partite set. Hence, $edim(K_{r,t}) \geq r+t-2$.

For the contrary, take $v \in V$, $u \in U$ and let $S = V(K_{r,t}) \setminus \lbrace v, u \rbrace$. It can be easily checked that $S$ is an edge metric generator. Therefore $edim(K_{r,t}) \leq r+t-2$ and the equality follows.
\end{proof}

Another family of graphs with equality on its values for metric dimension and edge metric dimension are the tree graphs. Since we already know that the edge metric dimension of a path is $1$, we only consider trees that are not paths and compute the value of its edge metric dimension. To this end, we need the following terminology from \cite{landmarks}. Let $T=(V,E)$ be a tree and let $v \in V$. Define the equivalence relation $R_v$ in the following way: for every two edges $e,f$ we let $e R_v f$ if and only if there is a path in $T$ including $e$ and $f$ that does not have $v$ as an internal vertex. The subgraphs induced by the edges of the equivalence classes of $E$ are called the bridges of $T$ relative to $v$. Furthermore, for each vertex $v \in V$, the legs at $v$ are the bridges which are paths. We denote by $l_v$ the number of legs at $v$.

We remark that the edge metric dimension for a tree can be computed in linear time. However, the algorithm to obtain an edge metric generator is the same as for the standard metric dimension (see \cite{landmarks}). For the sake of completeness, in the following proof we briefly describe the procedure anyway.

\begin{remark}
Let $T =(V,E)$ be a tree which is not a path. Then
$$edim(T) = dim(T) = \sum_{v \in V, \ l_v > 1} (l_v - 1).$$
\end{remark}

\begin{proof}
Let $v$ be a vertex of $T$ such that $l_v > 1$ and let $S$ be an edge metric generator. Suppose that at least two of the $v$'s legs do not contain an element of $S$. Then the edges incident to $v$ in those legs without an element of $S$ have the same distance to every element of $S$, a contradiction. Therefore, at least $l_v - 1$ legs of $v$ must contain an element of $S$. Since $T$ is not a path, the legs corresponding to different vertices are disjoint and therefore, $edim(T) \geq \sum_{v \in V, \ l_v > 1} (l_v - 1)$. \\

\noindent
For the contrary, we shall construct an edge metric generator $S'$ for an arbitrary tree (which is not a path) in the following way:
\begin{itemize}
\item Compute $l_v$ for each vertex $v$,
\item For every vertex $v$ with $l_v > 1$, put in the set $S'$ all but one of the leaves associated with the legs of $v$.
\end{itemize}

\noindent
Similarly as in \cite{landmarks} we deduce that such $S'$ is an edge metric generator. Therefore, $edim(T) \leq \sum_{v \in V, \ l_v > 1} (l_v - 1)$ and the equality $edim(T) = \sum_{v \in V, \ l_v > 1} (l_v - 1)$ follows. Finally, since the same formula is used to calculate the metric dimension of a tree which is not a path (see \cite{landmarks}), the proof is completed.
\end{proof}

The Cartesian product of two graphs $G$ and $H$ is the graph $G\Box H$, such that $V(G\Box H)=\{(a,b)\;:\;a\in V(G),\;b\in V(H)\}$ and two vertices $(a,b)$ and $(c,d)$ are adjacent in $G\Box H$ if and only if, either
\begin{itemize}
\item $a=c$ and $bd\in E(H)$, or
\item $b=d$ and $ac\in E(G)$.
\end{itemize}
Let $h \in V(H)$. We refer to the set $V(G)\times \{h\}$ as a $G$-layer. Similarly $\{g\} \times V(H)$, $g \in V(G)$ is an $H$-layer. When referring to a specific $G$ or $H$ layer, we denote them by $G^h$ or $^gH$, respectively. Obviously, the subgraph induced by a $G$-layer or by an $H$-layer is isomorphic to $G$ or $H$, respectively. Next we give the value of the edge metric dimension of the grid graph, which is the Cartesian product of two paths $P_r$ and $P_t$ with $r$ and $t$ vertices, respectively.

\begin{proposition}
Let $G$ be the grid graph $G=P_r\Box P_t$, with $r\geq t\geq 2$. Then $edim(G)=dim(G)=2$.
\end{proposition}

\begin{proof}
Since $G$ is not a path, by Remark \ref{rem-path-1}, it follows that $edim(G)\geq 2$.
For easily computing distances, let embed $G$ into $\ZZ^2$. Hence, each vertex can be represented as an ordered pair of its coordinates $(x,y)$. We embed $G$ into $\ZZ^2$ such that $(0,0),(r-1,0),(0,t-1), (r-1,t-1)$ are corner vertices of $G$.

\begin{figure}[h]
\centering
\begin{tikzpicture}[scale=.8, transform shape]

\draw [->](0,-1.2)--(0,5.2) node[right]{$y$};
  \draw [->](-1.2,0)--(6.2,0) node[right]{$x$};

\foreach \x/\xtext in {-1/-1, 1/1, 2/2, 3/3, 4/4, 5/5}
{\draw (\x cm,0pt ) -- (\x cm,-3pt ) node[anchor=north] {$\xtext$};}
\foreach \y/\ytext in {-1/-1, 1/1, 2/2, 3/3, 4/4}
{\draw (0pt,\y cm) -- (-3pt ,\y cm) node[anchor=east] {$\ytext$};}

\node [draw, shape=circle, label={[label distance=0cm]225:0}]  (66) at (0,0) {} ;

\foreach \x in {0,...,5}
    \foreach \y in {0,...,4}
       {\pgfmathtruncatemacro{\label}{\x - 5 *  \y +21}
       \node [draw, shape=circle]  (\x\y) at (1*\x,1*\y) {} ;}

  \foreach \x in {0,...,5}{
    \foreach \y in {0,...,3}
      \pgfmathtruncatemacro{\yi}{\y + 1}
      \draw (\x\y)--(\x\yi) ;
    \foreach \y in {0,...,4}
      \pgfmathtruncatemacro{\xi}{\x + 1}
      \ifthenelse{\NOT\x=5}{
      \draw (\x\y)--(\xi\y) ;
      }{};
  }
\end{tikzpicture}
\caption{Embedding of a grid graph $G=P_6 \Box P_5$ into $\ZZ^2$.}
\label{example}
\end{figure}
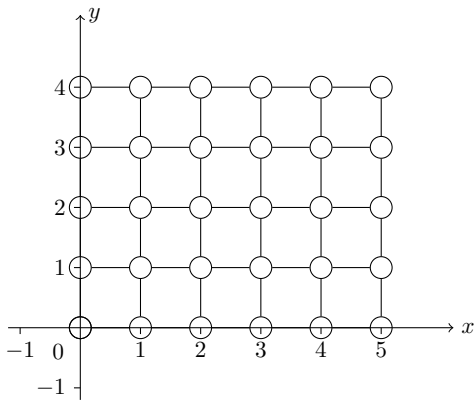

Let $S$ be the set containing the two vertices $a=(0,0)$ and $b=(r-1,0)$. We shall prove that such $S$ is an edge metric generator for the graph $G$. To this end, we notice that the distance between any two vertices in such representation of $G$ is $d((x_1,y_1),(x_2,y_2))=|x_1-x_2|+|y_1-y_2|$. We assume that each edge is an unordered pair of its endpoints $e=(x_1,y_1)(x_2,y_2)$ and always write such edge considering that $x_1 \leq x_2$ and $y_1 \leq y_2$. This implies that the distances from the edge $e=(x_1,y_1)(x_2,y_2)$ to the vertices $a$ and $b$ are $d(a,e)=x_1+y_1$ and $d(b,e)=r-1-x_2+y_1$, respectively.

Toward a contradiction, suppose that there exist two different edges $e=(x_1,y_1)(x_2,y_2)$ and $f=(w_1,z_1)(w_2,z_2)$ with the same distances to the vertices $a$ and $b$. This implies two equalities:
$$x_1+y_1=w_1+z_1$$
$$ r-1-x_2+y_1 = r-1-w_2+z_1 \iff y_1-z_1=x_2-w_2. $$

Thus, it follows that $x_1+x_2=w_1+w_2$. In both cases $x_1=x_2$ or $x_1=x_2-1$ we get $x_1=w_1$ and $x_2=w_2$.
The equality $x_1=w_1$ together with $x_1+y_1=w_1+z_1$ implies that $y_1=z_1$.
So, we deduce that both $y_2$ and $z_2$ can get values $y_1$ or $y_1+1$. If they get different values, then one of the edges $e$ or $f$ does not represent an edge. We finally get $e=f$, which is a contradiction.

It is already known from \cite{landmarks} that the metric dimension of grid graphs equals two. Thus, we finally get $dim(G)=edim(G)$ and the proof is completed.
\end{proof}

\subsection{Graphs for which $dim(G)<edim(G)$}\label{dim-less-edim}

The \emph{wheel graph} $W_{1,n}$  is the graph obtained from a cycle $C_n$ and the trivial graph $K_1$ by adding all the edges between the vertex of $K_1$ and every vertex of $C_n$. It is known (see \cite{buczkowski}) that

$$dim(W_{1,n})= \left\{ \begin{array}{ll} 3, & n = 3, 6,\\ 2, & n=4,5,  \\ \left\lfloor\frac{2n+2}{5} \right\rfloor, & n \geq 6.  \end{array} \right.$$

\noindent
In the next proposition we consider the edge metric dimension of wheel graphs and observe it is strictly larger than the value for the metric dimension, except in the case $W_{1,3}$.

\begin{proposition}
Let $W_{1,n}$ be a wheel graph. Then
$$edim(W_{1,n})= \left\{ \begin{array}{ll} n, & n = 3, 4,\\ n-1, & n \geq 5.  \end{array} \right.$$
\end{proposition}

\begin{proof}
If $n=3$ or $n=4$, then the proof is straightforward. Let $n \geq 5$ and $V(W_{1,n}) = \lbrace x, g_1, g_2, \ldots, g_n \rbrace$, where the vertex $x$ has degree $n$ and the vertices $g_1, \ldots, g_n$ induce a cycle $C_n$. Set $S = \lbrace g_1, g_2, \ldots, g_{n-1} \rbrace$. We show that $S$ is an edge metric generator. Let $e$ be an edge of $W_{1,n}$. Consider the following cases:
\begin{itemize}
\item If $e = g_ig_{i+1}$ for some $i \in \lbrace 1, \ldots, n-2 \rbrace $, then $e$ has distance $0$ to $g_i$ and $g_j$ and distance $1$ or $2$ to every other vertex in $S$.

\item If $e= g_{n-1}g_n$, then $e$ has distance $0$ to $g_{n-1}$, distance $1$ to $g_1$ and $g_{n-2}$, and distance $2$ to every other vertex in $S$ (and since $n\geq 5$ there is at least one such vertex).

\item If $e=g_ng_1$, then $e$ has distance $0$ to $g_1$, distance $1$ to $g_{n-1}$ and $g_2$, and distance $2$ to every other vertex in $S$ (and since $n\geq 5$ there is at least one such vertex).

\item If $e=xg_i$ for some $i \in \lbrace 1, \ldots, n-1 \rbrace$, then $e$ has distance $0$ to $g_i$ and distance $1$ to every other vertex in $S$.

\item If $e=xg_n$, then $e$ has distance $1$ to every vertex in $S$.
\end{itemize}
Now, it clearly follows from the items above that the edge metric representations of any two distinct vertices of $W_{1,n}$ are different. Thus, $S$ is an edge metric generator and therefore, $edim(W_{1,n}) \leq n-1$.

On the other hand, assume that $S$ is a set of vertices without at least two distinct vertices $g_i,g_j$ of the set $\lbrace g_1, \ldots, g_n \rbrace$. Consider the edges $e = xg_i$ and $f=xg_j$. Notice that $e$ and $f$ have the same distance to every vertex in $S$ and so, $S$ is not an edge metric generator. Therefore, $edim(W_{1,n})\geq n-1$ and we are done.
\end{proof}

Similarly to the wheel graph, the \emph{fan graph} $F_{1,n}$ is the graph obtained from a path $P_n$ and the trivial graph $K_1$ by adding all the edges between the vertex of $K_1$ and every vertex of $P_n$.  \noindent
For the case of fan graphs it is known (see \cite{Caceres2005}) that
$$dim(F_{1,n})= \left\{ \begin{array}{ll} 1, & n = 1,\\ 2, & n=2,3,  \\ 3, & n = 6, \\ \left\lfloor\frac{2n+2}{5} \right\rfloor, & \mbox{otherwise}.  \end{array} \right.$$

By using an analogous procedure as in the case of wheel graphs, we can compute the edge metric dimension for fan graphs, which is again strictly larger than the value for the metric dimension with the exception of $F_{1,n}$ with $n\in \{1,2\}$. We omit the proof since it is quite similar to the one above in wheel graphs.

\begin{proposition}
Let $F_{1,n}$ be a fan graph. Then
$$edim(F_{1,n})=\left\{ \begin{array}{ll}
                                       n, & n = 1,2,3,\\
                                       n-1, & n \geq 4.
                                      \end{array}
                         \right.$$
\end{proposition}






\subsection{Graphs for which $edim(G)<dim(G)$}

According to the definition of layers in the Cartesian product of two graphs given in Subsection \ref{dim=edim}, we say that an edge $e\in E(G\Box H)$ is \emph{vertical}, if $e$ lies in a $^gH$-layer for some $g\in V(G)$. Similarly, $e\in E(G\Box H)$ is \emph{horizontal}, if $e$ lies in an $^hG$-layer for some $h\in V(H)$.

The value of the metric dimension of several families of Cartesian product graphs was obtained in \cite{Caceres2007}. For instance, there was proved that
$$dim(C_r\Box C_t)=\left\{\begin{array}{ll}
                           4, & \mbox{if $r \cdot t$ is even}, \\
                           3, & \mbox{otherwise}.
                         \end{array}
\right.$$
Next we show that for some particular cases of the torus graphs $C_r\Box C_t$, it follows that $edim(C_{r}\Box C_{t})<dim(C_{r}\Box C_{t})$.

\begin{theorem}\label{torus-4r-4t}
For any integers $r,t$, $edim(C_{4r}\Box C_{4t})=3$.
\end{theorem}

\begin{proof}
We assume that $V(C_{4r})=\{a_0,a_1,\dots,a_{4r-1}\}$ and $V(C_{4t})=\{b_0,b_1,\dots,b_{4t-1}\}$ and for short, let $G=C_{4r}\Box C_{4t}$. From now on, in this proof, all the operations with the subindexes of vertices of $C_{4r}$ and $C_{4t}$ are done modulo $4r$ and $4t$, respectively. Moreover, we assume that $a_ia_{i+1}\in E(C_{4r})$ and $b_jb_{j+1}\in E(C_{4t})$ for every $i\in \{1,\dots,r\}$ and $j\in \{1,\dots,t\}$, respectively. We shall prove that the set $S=\{(a_0,b_0),(a_0,b_{2t}),(a_r,b_t)\}$ is an edge metric generator for $G$. Let $e,f$ be any edges of $G$. We consider the following cases.

\noindent \textbf{Case 1:} $e$ is a horizontal edge and $f$ is a vertical edge.\\
Hence, without loss of generality we assume that the edges $e=(g_1,h)(g_2,h)$ and $f=(g,h_1)(g,h_2)$ satisfy that $g_1$ is closer to $a_0$ than $g_2$ and that $h_1$ is closer to $b_0$ than $h_2$. Thus, we have the following.
$$d_{G}(e,(a_0,b_0))=d_{C_{4t}}(h,b_0)+d_{C_{4r}}(g_1,a_0),$$
$$d_{G}(e,(a_0,b_{2t}))=d_{C_{4t}}(h,b_{2t})+d_{C_{4r}}(g_1,a_0),$$
$$d_{G}(f,(a_0,b_0))=d_{C_{4t}}(h_1,b_0)+d_{C_{4r}}(g,a_0),$$
$$d_{G}(f,(a_0,b_{2t}))=d_{C_{4t}}(h_2,b_{2t})+d_{C_{4r}}(g,a_0).$$
Suppose $d_{G}(e,(a_0,b_0))=d_{G}(f,(a_0,b_0))$ and $d_{G}(e,(a_0,b_{2t}))=d_{G}(f,(a_0,b_{2t}))$. So, from the equalities above we obtain
$$d_{C_{4t}}(h,b_0)+d_{C_{4r}}(g_1,a_0)=d_{C_{4t}}(h_1,b_0)+d_{C_{4r}}(g,a_0)$$ and $$d_{C_{4t}}(h,b_{2t})+d_{C_{4r}}(g_1,a_0)=d_{C_{4t}}(h_2,b_{2t})+d_{C_{4r}}(g,a_0).$$
Since $d_{C_{4t}}(h,b_0)+d_{C_{4t}}(h,b_{2t})=2t$ and $d_{C_{4t}}(h_1,b_0)+d_{C_{4t}}(h_2,b_{2t})=2t-1$, by adding the last two equalities we deduce that
$$2d_{C_{4r}}(g_1,a_0)=2d_{C_{4r}}(g,a_0)-1,$$
which is not possible, since the left side of the equality is an even number and the right side is odd. Thus, we have that $d_{G}(e,(a_0,b_0))\ne d_{G}(f,(a_0,b_0))$ or $d_{G}(e,(a_0,b_{2t}))\ne d_{G}(f,(a_0,b_{2t}))$. Equivalently, $e,f$ are distinguished by $(a_0,b_0)$ or by $(a_0,b_{2t})$.

\noindent \textbf{Case 2:} $e,f$ are vertical edges.\\
Similarly to the case above, without loss of generality, we assume that the edges $e=(x,h_1)(x,h_2)$ and $f=(y,h_3)(y,h_4)$ satisfy that $h_1$ is closer to $b_0$ than $h_2$ and that $h_3$ is closer to $b_0$ than $h_4$. Thus, we have the following.
$$d_{G}(e,(a_0,b_0))=d_{C_{4t}}(h_1,b_0)+d_{C_{4r}}(x,a_0),$$
$$d_{G}(e,(a_0,b_{2t}))=d_{C_{4t}}(h_2,b_{2t})+d_{C_{4r}}(x,a_0),$$
$$d_{G}(f,(a_0,b_0))=d_{C_{4t}}(h_3,b_0)+d_{C_{4r}}(y,a_0),$$
$$d_{G}(f,(a_0,b_{2t}))=d_{C_{4t}}(h_4,b_{2t})+d_{C_{4r}}(y,a_0).$$
Now, assume that $d_{G}(e,(a_0,b_0))=d_{G}(f,(a_0,b_0))$ and $d_{G}(e,(a_0,b_{2t}))=d_{G}(f,(a_0,b_{2t}))$. Thus, the four equalities above lead to
\begin{equation}\label{eq1}
d_{C_{4t}}(h_1,b_0)+d_{C_{4r}}(x,a_0)=d_{C_{4t}}(h_3,b_0)+d_{C_{4r}}(y,a_0),
\end{equation}
\begin{equation}\label{eq2}
d_{C_{4t}}(h_2,b_{2t})+d_{C_{4r}}(x,a_0)=d_{C_{4t}}(h_4,b_{2t})+d_{C_{4r}}(y,a_0).
\end{equation}
By adding these two equalities and by using the fact that $d_{C_{4t}}(h_1,b_0)+d_{C_{4t}}(h_2,b_{2t})=2t-1$ and $d_{C_{4t}}(h_3,b_0)+d_{C_{4t}}(h_4,b_{2t})=2t-1$,  we deduce that
$$d_{C_{4r}}(x,a_0)=d_{C_{4r}}(y,a_0).$$
Moreover, by using the equality above in the equalities (\ref{eq1}) and (\ref{eq2}), it follows that
$$d_{C_{4t}}(h_1,b_0)=d_{C_{4t}}(h_3,b_0),$$
$$d_{C_{4t}}(h_2,b_{2t})=d_{C_{4t}}(h_4,b_{2t}).$$
As a consequence of these three last relationships we notice that any two edges $e,f$ having the same distance to the vertices $(a_0,b_0)$ and $(a_0,b_{2t})$ satisfy one of the following situations:
\begin{itemize}
\item $e,f$ are symmetrical with respect to the $^{a_0} C_{4t}$-layer (see pairs of edges $(e_i,f_i)$, with $i\in \{1,\ldots,4\}$, drawn in Figure \ref{fig-vertical}),
\item $e,f$ are symmetrical with respect to the $C_{4r}\,^{b_0}$-layer or equivalently to the $C_{4r}\,^{b_{2t}}$-layer (see pairs of edges $(e_1,e_4)$, $(e_2,e_3)$, $(f_1,f_4)$, $(f_2,f_3)$, drawn in Figure \ref{fig-vertical}),
\item $e,f$ are symmetrical with respect to the vertex $(a_0,b_0)$ or equivalently to the vertex $(a_0,b_{2t})$ (see pairs of edges $(e_1,f_4)$, $(e_2,f_3)$, $(f_1,e_4)$, $(f_2,e_3)$ drawn in Figure \ref{fig-vertical}).
\end{itemize}

\begin{figure}[ht]
\centering
\begin{tikzpicture}[scale=.5]
  \foreach \x in {0,2,4,6,8,10,12,14,16,18,20,22}
  {
    \foreach \y in {0,2,4,6,8,10,12,14}
    {
      \node at (\x,\y) [draw, shape=circle,scale=.5] {};
    }
  }

  \foreach \x in {0,1,2,3,4,5,6}
  {
   \node at (\x+\x+10,-2) {$a_{\x}$};
  }
  \foreach \x in {7,8,9,10,11}
  {
   \node at (\x+\x-14,-2) {$a_{\x}$};
  }
  \foreach \y in {0,1,2,3,4,5,6,7}
  {
   \node at (-2,\y+\y) {$b_{\y}$};
  }

\draw (6,2.15)--(6,3.85);
   \node at (6.5,3) {$e_4$};
\draw (6,12.15)--(6,13.85);
   \node at (6.5,13) {$e_1$};
\draw (14,2.15)--(14,3.85);
   \node at (14.5,3) {$f_4$};
\draw (14,12.15)--(14,13.85);
   \node at (14.5,13) {$f_1$};
\draw (2,6.15)--(2,7.85);
   \node at (2.5,7) {$e_3$};
\draw (2,8.15)--(2,9.85);
   \node at (2.5,9) {$e_2$};
\draw (18,6.15)--(18,7.85);
   \node at (18.5,7) {$f_3$};
\draw (18,8.15)--(18,9.85);
   \node at (18.5,9) {$f_2$};

\draw (0.15,2)--(1.85,2);
   \node at (1,1.5) {$e_6$};
\draw (0.15,14)--(1.85,14);
   \node at (1,13.5) {$e_5$};
\draw (6.15,6)--(7.85,6);
   \node at (7,5.5) {$e_8$};
\draw (6.15,10)--(7.85,10);
   \node at (7,9.5) {$e_7$};
\draw (18.15,2)--(19.85,2);
   \node at (19,1.5) {$f_6$};
\draw (18.15,14)--(19.85,14);
   \node at (19,13.5) {$f_5$};
\draw (12.15,6)--(13.85,6);
   \node at (13,5.5) {$f_8$};
\draw (12.15,10)--(13.85,10);
   \node at (13,9.5) {$f_7$};

\node at (10,0) [circle, fill=black,scale=.5] {};
\node at (10,8) [circle, fill=black,scale=.5] {};
\node at (16,4) [circle, fill=black,scale=.5] {};
\end{tikzpicture}
\caption{A sketch of the graph $C_{12}\Box C_8$. Only some examples of edges have been drawn. Vertices in bold represent the edge metric generator.}\label{fig-vertical}
\end{figure}
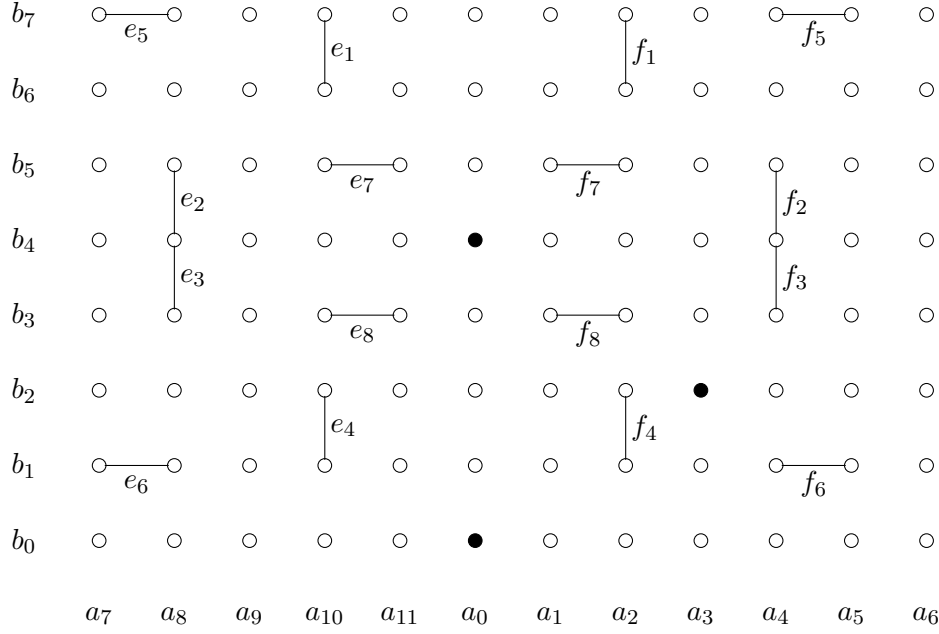

According to these items above and due to the fact that the cycles used to generate the graph $G$ have order $4r$ and $4t$, it is not difficult to notice that if two vertical edges are not distinguished by the vertices $(a_0,b_0)$ and $(a_0,b_{2t})$, then they are distinguished by the vertex $(a_r,b_t)$. For instance, assume that $e,f$ are symmetrical with respect to the $^{a_0} C_{4t}$-layer. So, without loss of generality assume that $e$ lies in a $^{a_i} C_{4t}$-layer with $i\in \{1,\ldots,2r-1\}$. Thus, $f$ lies in a $^{a_j} C_{4t}$-layer with $j\in \{2r+1,\ldots,4r-1\}$ (notice that neither $e$ nor $f$ lie in the $^{a_{2r}} C_{4t}$-layer since in such case $e=f$, which is not possible). Hence, it follows that
\begin{equation}\label{eq5}
d_G(e,(a_r,b_t))=d_G(e,(a_i,b_t))+d_{C_{4r}}(a_i,a_r)
\end{equation}
and
\begin{equation}\label{eq55}
d_G(f,(a_r,b_t))=d_G(f,(a_j,b_t))+d_{C_{4r}}(a_j,a_r).
\end{equation}
Note that $d_G(e,(a_i,b_t))=d_G(f,(a_j,b_t))$, since $e,f$ are symmetrical with respect to the $^{a_0} C_{4t}$-layer. Moreover, it clearly happens that $d_{C_{4r}}(a_i,a_r)<d_{C_{4r}}(a_j,a_r)$, since $d_{C_{4r}}(a_i,a_0)=d_{C_{4r}}(a_j,a_0)$. Thus, equalities given in (\ref{eq5}) and (\ref{eq55}) lead to $d_G(e,(a_r,b_t))\ne d_G(f,(a_r,b_t))$.

\noindent \textbf{Case 3:} $e,f$ are horizontal edges.\\
The procedure in this case is relatively similar to that in Case 2. As such, we assume that the edges $e=(g_1,y)(g_2,y)$ and $f=(g_3,z)(g_4,z)$ satisfy that $g_1$ is closer to $a_0$ than $g_2$ and that $g_3$ is closer to $a_0$ than $g_4$. Thus,
$$d_{G}(e,(a_0,b_0))=d_{C_{4t}}(y,b_0)+d_{C_{4r}}(g_1,a_0),$$
$$d_{G}(e,(a_0,b_{2t}))=d_{C_{4t}}(y,b_{2t})+d_{C_{4r}}(g_1,a_0),$$
$$d_{G}(f,(a_0,b_0))=d_{C_{4t}}(z,b_0)+d_{C_{4r}}(g_3,a_0),$$
$$d_{G}(f,(a_0,b_{2t}))=d_{C_{4t}}(z,b_{2t})+d_{C_{4r}}(g_3,a_0).$$
As before, we assume that $d_{G}(e,(a_0,b_0))=d_{G}(f,(a_0,b_0))$ and $d_{G}(e,(a_0,b_{2t}))=d_{G}(f,(a_0,b_{2t}))$. Thus, the four equalities above lead to
\begin{equation}\label{eq3}
d_{C_{4t}}(y,b_0)+d_{C_{4r}}(g_1,a_0)=d_{C_{4t}}(z,b_0)+d_{C_{4r}}(g_3,a_0),
\end{equation}
\begin{equation}\label{eq4}
d_{C_{4t}}(y,b_{2t})+d_{C_{4r}}(g_1,a_0)=d_{C_{4t}}(z,b_{2t})+d_{C_{4r}}(g_3,a_0).
\end{equation}
By adding these two equalities and by using the fact that $d_{C_{4t}}(y,b_0)+d_{C_{4t}}(y,b_{2t})=2t$ and $d_{C_{4t}}(z,b_0)+d_{C_{4t}}(z,b_{2t})=2t$,  we deduce that
$$d_{C_{4r}}(g_1,a_0)=d_{C_{4r}}(g_3,a_0).$$
Also, by using the equality above in the equalities (\ref{eq3}) and (\ref{eq4}), we have
$$d_{C_{4t}}(y,b_0)=d_{C_{4t}}(z,b_0),$$
$$d_{C_{4t}}(y,b_{2t})=d_{C_{4t}}(z,b_{2t}).$$
Thus, we deduce that for any two edges $e,f$ having the same distance to the vertices $(a_0,b_0)$ and $(a_0,b_{2t})$ one of the following situations is satisfied:
\begin{itemize}
\item $e,f$ are symmetrical with respect to the $^{a_0} C_{4t}$-layer (see pairs of edges $(e_i,f_i)$, with $i\in \{5,\ldots,8\}$, drawn in Figure \ref{fig-vertical}),
\item $e,f$ are symmetrical with respect to the $C_{4r}\,^{b_0}$-layer or equivalently to the $C_{4r}\,^{b_{2t}}$-layer (see pairs of edges $(e_5,e_6)$, $(e_7,e_8)$, $(f_5,f_6)$ and $(f_7,f_8)$  drawn in Figure \ref{fig-vertical}),
\item $e,f$ are symmetrical with respect to the vertex $(a_0,b_0)$ or equivalently to the vertex $(a_0,b_{2t})$ (see pairs of edges $(e_5,f_6)$, $(e_6,f_5)$, $(e_7,f_8)$ and $(e_8,f_7)$ drawn in Figure \ref{fig-vertical}).
\end{itemize}
By using a similar reasoning like in Case 2, we deduce that if two horizontal edges are not distinguished by the vertices $(a_0,b_0)$ and $(a_0,b_{2t})$, then they are distinguished by the vertex $(a_r,b_t)$.

As a consequence of the three cases above we obtain that $S$ is an edge metric generator, which leads to $edim(C_{4r}\Box C_{4t})\le 3$. Now, consider two distinct vertices $(a,b),(c,d)\in V(C_{4r}\Box C_{4t})$. Notice that there are always two incident edges with $(a,b)$ (or with $(c,d)$), such that they are not distinguished by $(a,b)$ nor by $(c,d)$. Therefore, $edim(C_{4r}\Box C_{4t})>2$, which completes the proof.
\end{proof}

\subsection{Realization of the edge metric dimension versus the metric dimension}\label{subsect-realiz}

Since it is possible to find classes of graphs $G$ such that $dim(G)=edim(G)$, $dim(G)<edim(G)$ or $edim(G)<dim(G)$, the realization question stated at the beginning of this section (concerning the triplet $r,t,n$: metric dimension, edge metric dimension and order, respectively) must be dealt with by separating these three possibilities above.

The case $dim(G)=edim(G)$ is realizable by complete or tree graphs for instance. That is, the triplet $n-1,n-1,n$ is realizable by a complete graph $K_n$ and the triplet $r,r,n$ with $1\le r\le n-2$ is realizable by a tree $T$ with $r+1$ leaves obtained from a star $S_{1,n-1}$ by removing $n-1-r$ edges of $S_{1,n-1}$ and subdividing one of the remaining edges with $n-1-r$ vertices. Clearly the order of such $T$ is $n$ and it is straightforward to observe that $dim(T)=edim(T)=r$. Notice that the particular case $r=1$ is given by the path graph $P_n$, which is also obtained as described above.

We next continue with the case $dim(G)<edim(G)$. To this end, we need the following family $\mathcal{F}$ of graphs.
We begin with a star graph $S_{1,b}$, $b\ge 2$, and the graph $G_1=K_1+(\bigcup_{i=1}^aK_2)$, $a\ge 1$, where the operator $(+)$ represents the join graph\footnote{The \emph{join graph} $G+H$ of the graphs $G$ and $H$ is a graph obtained from $G$ and $H$ by adding all the possible edges between a vertex of $G$ and a vertex of $H$.}. Then, to obtain a graph $G_{a,b,c}\in \mathcal{F}$, we choose a path $P_c$ of order $c$ and join by an edge one leaf of $P_c$ with the center of $G_1$, and the other leaf with the center of the star $S_{1,b}$. We shall make the assumption that $c$ could be equal to zero, and in such case the action above (adding the path $P_c$) is understood as adding an edge between the centers of $G_1$ and $S_{1,b}$.  See Figure \ref{family-F} for an example.

\begin{figure}[ht]
\centering
\begin{tikzpicture}[scale=.7, transform shape]

\node [draw, shape=circle] (a1) at (0,1) {};
\node [draw, shape=circle] (a2) at (0,3) {};
\node [draw, shape=circle] (a3) at (1,0) {};
\node [draw, shape=circle] (a4) at (2,2) {};
\node [draw, shape=circle] (a5) at (1,4) {};
\node [draw, shape=circle] (a6) at (3,0) {};
\node [draw, shape=circle] (a7) at (3,4) {};
\node [draw, shape=circle] (a8) at (4,2) {};
\node [draw, shape=circle] (a9) at (6,2) {};
\node [draw, shape=circle] (a10) at (8,2) {};

\node [draw, shape=circle] (b4) at (9,0) {};
\node [draw, shape=circle] (b5) at (9,4) {};
\node [draw, shape=circle] (b6) at (10,2) {};
\node [draw, shape=circle] (b7) at (11,0) {};
\node [draw, shape=circle] (b8) at (11,4) {};
\node [draw, shape=circle] (b9) at (12,1) {};
\node [draw, shape=circle] (b10) at (12,3) {};

\draw(a4)--(a1)--(a2)--(a4)--(a3)--(a6)--(a4)--(a5)--(a7)--(a4)--(a8)--(a9)--(a10);
\draw(a10)--(b6)--(b4);
\draw(b7)--(b6)--(b5);
\draw(b8)--(b6)--(b9);
\draw(b10)--(b6);

\end{tikzpicture}
\caption{The graph $G_{3,6,3}$.}\label{family-F}
\end{figure}
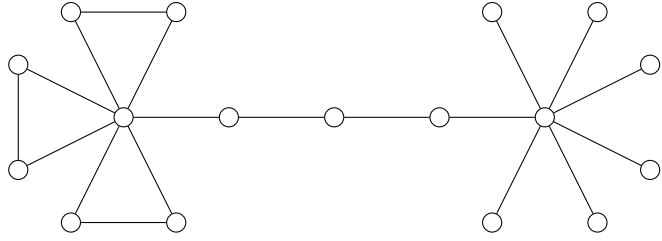

Observe that a graph $G_{a,b,c}\in \mathcal{F}$ has order $2a+b+c+2$. Next we compute $dim(G_{a,b,c})$ and $edim(G_{a,b,c})$ for any $G_{a,b,c}\in \mathcal{F}$.

\begin{remark}\label{graph-F}
Let $G_{a,b,c}\in \mathcal{F}$. Then $dim(G_{a,b,c})=a+b-1$ and $edim(G_{a,b,c})=2a+b-2$.
\end{remark}

\begin{proof}
Let $S$ be a metric basis of $G_{a,b,c}$. Notice that any two distinct vertices of the star $S_{1,b}$ have the same distance to any other vertex of $G_{a,b,c}$. Moreover, any two adjacent vertices of $G_1$ different from the center have the same distance to any other vertex of $G_{a,b,c}$. As a consequence of these two observations, we deduce that $S$ must contain at least $b-1$ vertices of the star $S_{1,b}$ and at least $a$ vertices of $G_1$. Thus $dim(G_{a,b,c})\ge a+b-1$. On the other hand, it is straightforward to observe that a set composed by $b-1$ leaves of the star $S_{1,b}$ and one vertex of each graph $K_2$ used to generate $G_1$ is a metric generator for $G_{a,b,c}$. Therefore $dim(G_{a,b,c})\le a+b-1$ and the first equality follows.

Now, let $S'$ be an edge metric basis of $G_{a,b,c}$. We observe that any two edges joining the center of $G_1$ with any other vertex in $G_1$ have the same distance to every other vertex of $G_{a,b,c}$. Also, any two edges of the star $S_{1,b}$ have the same distance to every other vertex of $G_{a,b,c}$. Thus, we deduce that $S'$ must contains at least $b-1$ vertices of the star $S_{1,b}$ and $2a-1$ vertices of $G_1$. So $edim(G_{a,b,c})\ge 2a+b-2$. It is again straightforward to observe that a set composed by $b-1$ leaves of the star $S_{1,b}$ and all but two vertices of $G_1$ (the center and other extra vertex) is an edge metric generator for $G_{a,b,c}$. Therefore, $edim(G_{a,b,c})\le 2a+b-2$ and the second equality follows.
\end{proof}

By using the family above we partially solve the realization question regarding the triplet order, $dim(G)$, $edim(G)$ whenever $dim(G)<edim(G)$. We first observe that the triplet $1,t,n$ with $t\ge 2$ is not realizable for any graph $G$, since $dim(G)=1$ if and only if $G$ is a path $P_n$ and $edim(P_n)=1$. In our next theorem we consider that $2r\le n-2$, otherwise the theorem would be stated for any $r,t,n$ such that $2\le r\le t\le n-2$.

\begin{theorem}
For any $r,t,n$ such that $2\le r\le t\le 2r\le n-2$, there exists a connected graph $G$ of order $n$ such that $dim(G)=r$ and $edim(G)=t$.
\end{theorem}

\begin{proof}
We first deal with the case $t=2r$. Let $G_{r,n}$ be the graph obtained as follows. We begin with the join graph $G'=K_1+(K_1\cup
(\bigcup_{i=1}^rK_2))$. Then we add a path of order $n-2r-2$ and join with an edge one of its leaves with the unique vertex of $G'$ of degree one. See Figure \ref{graphs-realiz} for an example. Clearly, $G_{r,n}$ has order $n-2r-2+2r+2=n$. Also, it is not difficult to see that any metric generator needs $r$ vertices and that any edge metric generator needs $2r=t$ vertices. Thus, it follows $dim(G_{r,n})=r$ and $edim(G_{r,n})=t$. Since $n-2r-2\ge 0$ and $t=2r$ we get that $t\le n-2$ and we are done for this case.

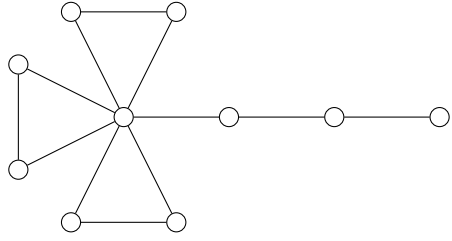
\begin{figure}[ht]
\centering
\begin{tikzpicture}[scale=.7, transform shape]

\node [draw, shape=circle] (a1) at (0,1) {};
\node [draw, shape=circle] (a2) at (0,3) {};
\node [draw, shape=circle] (a3) at (1,0) {};
\node [draw, shape=circle] (a4) at (2,2) {};
\node [draw, shape=circle] (a5) at (1,4) {};
\node [draw, shape=circle] (a6) at (3,0) {};
\node [draw, shape=circle] (a7) at (3,4) {};
\node [draw, shape=circle] (a8) at (4,2) {};
\node [draw, shape=circle] (a9) at (6,2) {};
\node [draw, shape=circle] (a10) at (8,2) {};

\draw(a4)--(a1)--(a2)--(a4)--(a3)--(a6)--(a4)--(a5)--(a7)--(a4)--(a8)--(a9)--(a10);
\end{tikzpicture}
\caption{The graph $G_{3,10}$.}\label{graphs-realiz}
\end{figure}

Now on we assume $2\le r\le t \le 2r-1\le n-2$. We consider a graph $G_{x,y,z}\in \mathcal{F}$. From Remark \ref{graph-F} we know that $G_{x,y,z}$ has order $2x+y+z+2$ and satisfies that $dim(G_{x,y,z})=x+y-1$ and $edim(G_{x,y,z})=2x+y-2$. Since we are looking for a graph $G$ of order $n$ such that  $dim(G)=r$ and $edim(G)=t$, we must find a graph $G_{x,y,z}\in \mathcal{F}$ for some $x,y,z$ that will satisfy the following system of linear equations.
$$\begin{array}{r}
  2x+y+z+2=n \\
  x+y-1=r \\
  2x+y-2=t
\end{array}$$
We can easily compute that such system has solution $x=t-r+1$, $y=2r-t$ and $z=n-t-4$ (note that these values represent integer numbers). Since the graph $G_{x,y,z}\in \mathcal{F}$ satisfies that $x\ge 1$, $y\ge 2$ and $z\ge 0$, we get that $t-r+1\ge 1$, $2r-t\ge 2$ and $n-t-4\ge 0$. Thus, it follows that $t\ge r$, $t\le 2r-2$ and $t\le n-4$.

According to this, only the following cases remain, (1): $t=2r-1\le n-2$ or (2): ($2\le r\le t\le 2r-2$ and $t\in \{n-3,n-2\}$). Assume $t=2r-1\le n-2$. Consider the graph $G_{r}$ obtained as follows. We begin with the graph $G''=K_1+ (\bigcup_{i=1}^rK_2)$. Then we add a path of order $n-2r-1\ge 0$ (the case $n-2r-1=0$ means that we do not add any path and, clearly $n=2r+1$) and join by an edge a leaf of such path with one non central vertex of $G''$. See Figure \ref{graphs-realiz-3} for an example. It is straightforward to observe that $G_r$ has order $n-2r-1+2r+1=n$ and satisfies that $dim(G_r)=r$ and that $edim(G_r)=2r-1=t$. Since $n-2r-1\ge 0$ and $t=2r-1$ we deduce that $t\le n-2$.

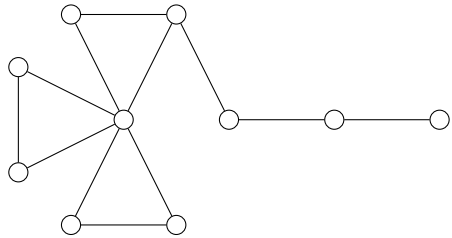
\begin{figure}[ht]
\centering
\begin{tikzpicture}[scale=.7, transform shape]

\node [draw, shape=circle] (a1) at (0,1) {};
\node [draw, shape=circle] (a2) at (0,3) {};
\node [draw, shape=circle] (a3) at (1,0) {};
\node [draw, shape=circle] (a4) at (2,2) {};
\node [draw, shape=circle] (a5) at (1,4) {};
\node [draw, shape=circle] (a6) at (3,0) {};
\node [draw, shape=circle] (a7) at (3,4) {};
\node [draw, shape=circle] (a8) at (4,2) {};
\node [draw, shape=circle] (a9) at (6,2) {};
\node [draw, shape=circle] (a10) at (8,2) {};

\draw(a4)--(a1)--(a2)--(a4)--(a3)--(a6)--(a4)--(a5)--(a7)--(a4);
\draw(a7)--(a8)--(a9)--(a10);
\end{tikzpicture}
\caption{The graph $G_{4}$.}\label{graphs-realiz-3}
\end{figure}

Finally, we assume $2\le r\le t\le 2r-2$ with $t\in \{n-3,n-2\}$. First suppose that $t=n-3$. Consider the graph $G'_{r,t}$ given by the join graph $K_1+[(\bigcup_{i=1}^{2r-t+1}K_1)\cup(\bigcup_{i=1}^{t-r}K_2)]$ and adding a pendant vertex to one of its vertices of degree one. See Figure \ref{graphs-realiz-2} (a) for an example. It is straightforward to observe that $G'_{r,t}$ has order $2(t-r)+2r-t+1+2=t+3=n$. Also, we can note that $dim(G'_{r,t})=2r-t+1+t-r-1=r$ and that $edim(G'_{r,t})=2r-t+1+2(t-r)-1=t$.

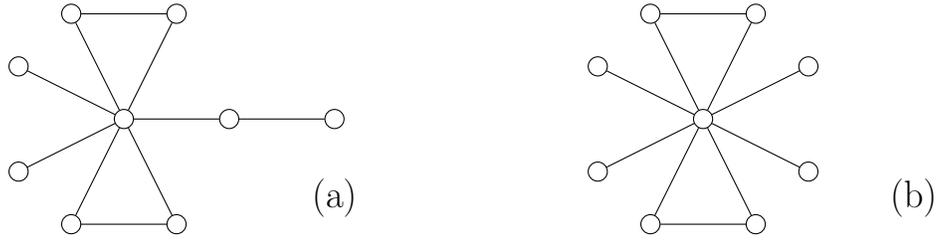
\begin{figure}[ht]
\centering
\begin{tikzpicture}[scale=.7, transform shape]

\node [draw, shape=circle] (a1) at (0,1) {};
\node [draw, shape=circle] (a2) at (0,3) {};
\node [draw, shape=circle] (a3) at (1,0) {};
\node [draw, shape=circle] (a4) at (2,2) {};
\node [draw, shape=circle] (a5) at (1,4) {};
\node [draw, shape=circle] (a6) at (3,0) {};
\node [draw, shape=circle] (a7) at (3,4) {};
\node [draw, shape=circle] (a8) at (4,2) {};
\node [draw, shape=circle] (a9) at (6,2) {};

\node [draw, shape=circle] (b2) at (11,1) {};
\node [draw, shape=circle] (b3) at (11,3) {};
\node [draw, shape=circle] (b4) at (12,0) {};
\node [draw, shape=circle] (b5) at (12,4) {};
\node [draw, shape=circle] (b6) at (13,2) {};
\node [draw, shape=circle] (b7) at (14,0) {};
\node [draw, shape=circle] (b8) at (14,4) {};
\node [draw, shape=circle] (b9) at (15,1) {};
\node [draw, shape=circle] (b10) at (15,3) {};

\draw(a1)--(a4)--(a2);
\draw(a4)--(a3)--(a6)--(a4)--(a5)--(a7)--(a4)--(a8)--(a9);
\draw(b3)--(b6)--(b4)--(b7)--(b6)--(b5)--(b8)--(b6)--(b9);
\draw(b10)--(b6)--(b2);

\draw (6,0) node[above] {\huge{(a)}};
\draw (17,0) node[above] {\huge{(b)}};

\end{tikzpicture}
\caption{The graph $G'_{4,6}$ (a) and the graph $G''_{5,7}$ (b).}\label{graphs-realiz-2}
\end{figure}

Now suppose that $t=n-2$. In such case we use a similar construction as above. Consider the graph $G''_{r,t}$ given by the join graph $K_1+[(\bigcup_{i=1}^{2r-t+1}K_1)\cup(\bigcup_{i=1}^{t-r}K_2)]$. See Figure \ref{graphs-realiz-2} (b) for an example. It is straightforward to observe that $G''_{r,t}$ has order $2(t-r)+2r-t+1+1=t+2=n$. Moreover, it is also satisfied that $dim(G''_{r,t})=2r-t+1+t-r-1=r$ and that $edim(G''_{r,t})=2r-t+1+2(t-r)-1=t$ and we are done for this case, which completes the whole proof.
\end{proof}

As a consequence of the theorem above, one could think that for any graph $G$ it follows that $edim(G)\le 2 dim(G)$. However, this is not true, which can be seen from the next example.

\begin{example}
Let us take the wheel graph $W_{1,6}$. In Section \ref{dim-less-edim} we recall the formulae for the metric dimension $($from {\em \cite{buczkowski}}$)$ and compute the edge metric dimension of wheel graphs, which gives for instance, $edim(W_{1,6})=5$ and $dim(W_{1,6})=2$. Thus, it follows that $edim(W_{1,6}) > 2dim(W_{1,6})$.
\end{example}

Other similar examples can be easily presented for wheels or fan graphs of higher order. Moreover, we also observe that the difference between edge metric dimension and metric dimension can be as large as possible.

\begin{proposition}
For any integer $q\ge 1$, there exists a connected graph $G$ such $edim(G)-dim(G)\ge q$.
\end{proposition}

\begin{proof}
The result can be obtained by using the wheel or fan graphs. For instance, from Subsection \ref{dim-less-edim} we know that for every $n\ge 6$ it follows that $dim(W_{1,n})=\left\lfloor\frac{2n+2}{5}\right\rfloor$ and that $edim(W_{1,n})=n-1$. Thus, by taking a wheel graph $W_{1,n}$ such that $n\ge \frac{5q+2}{3}$ we deduce that $n-1-\left\lfloor\frac{2n+2}{5}\right\rfloor\ge q$.
\end{proof}

According to the results obtained until here in this subsection for the case $dim(G)<edim(G)$, it remains to complete the realization of the triplet $r,t,n$ for the case $r\ge 2$ and $t>2r$ (if $2r<n-2$). Thus, we point out the following open problem. Is it possible to find a graph $G$ of order $n$ such that $dim(G)=r$ and $edim(G)=t$ for any integers $r,t,n$ with $r\ge 2$ and $2r<t\le n-2$?

Finally, we analyze the realizability of graphs $G$ for which $edim(G)<dim(G)$. In contrast with the other possibility $dim(G)<edim(G)$, it seems that given a triplet of integers $r,t,n$ with $2\le t<r\le n-2$, it is quite a challenging problem to provide a connected graph $G$ of order $n$ such that $dim(G)=r$ and $edim(G)=t$. From our results (only the Theorem \ref{torus-4r-4t}), we know that if $r=4$ and $t=3$, then for any $n=16k$ for some integer $k\ge 1$, it is possible to provide a graph satisfying the conditions above. On the contrary, we have not found any other example in which this is also satisfied and we post the following question. Given any three integers $r,t,n$ with $2\le t<r\le n-2$: Is it possible to construct a connected graph $G$ of order $n$ such that $dim(G)=r$ and $edim(G)=t$? Another approach could be related to finding a possible bound for $edim(G)$ in terms of $dim(G)$ for any connected graph $G$, under the supposition that $edim(G)<dim(G)$. For instance, if $G$ is the torus graph $C_{4r}\Box C_{4t}$, then $3=edim(G)=4-1=dim(G)-1$. In this sense: Is there a constant $c$ such that $edim(G)\le dim(G)-c$ for any connected graph $G$?

\section{Complexity issues}

Once studied some relationships between the edge metric dimension and the standard metric dimension, it is natural to think how much computationally difficult is the problem of computing the edge metric dimension of a graph. The decision problem concerning the metric dimension of a graph is already known as one of the classical NP-complete problems presented in the book \cite{garey} (a formal proof of it appeared in \cite{landmarks}). In this sense, it is natural to think that the similar problem for the edge metric dimension is also NP-complete, and one could think that an analogous result to that presented in \cite{landmarks} will immediately produce such conclusion. However, in concordance with other facts mentioned above, this is not the case. Indeed, proving the NP-completeness of our problem requires a harder working, although the reduction uses the 3-SAT problem, as in the case of the metric dimension proof of \cite{landmarks}. From now on, in this section we show that the problem of finding the edge metric dimension of an arbitrary connected graph is NP-hard. We first deal with the following decision problem.

$$\begin{tabular}{|l|}
  \hline
  \mbox{EDGE METRIC DIMENSION PROBLEM (EDIM problem for short)}\\
  \mbox{INSTANCE: A connected graph $G$ of order $n\ge 3$ and an integer $1\le r\le n-1$.}\\
  \mbox{QUESTION: Is $edim(G)\le r$?}\\
  \hline
\end{tabular}$$\\
To study the complexity of the problem above we make a reduction from the 3-SAT problem, which is one of the most classical problems known as NP-complete. For more information on this problem, and in NP-completeness reductions in general, we suggest \cite{garey}.

\begin{theorem}\label{theorem:NP-complete}
The EDIM problem is NP-complete.
\end{theorem}

\begin{proof}
The problem is easily seen to be in NP. For a set of vertices $S$ guessed by a nondeterministic algorithm for the problem, one need to check that this is an edge metric generator. This can be done in polynomial time by calculating the distances from vertices to edges and checking that all pairs of edges have different distance vectors with respect to the set $S$. We now describe a polynomial transformation of the 3-SAT problem to the EDIM problem.

Consider an arbitrary input of the 3-SAT problem, a collection $C=\{c_1, c_2, \ldots, c_m\}$ of clauses over a finite set $U=\{u_1,u_2,\ldots,u_n\}$ of Boolean variables. We shall construct a connected graph $G=(V,E)$, such that setting a positive integer $r \leq |V|$, the graph $G$ has an edge metric generator of size $r$ or less if and only if $C$ is satisfiable. The construction will be made up of several components augmented by some additional edges for communicating between various components.

For each variable $u_i \in U$ we construct a truth-setting component $X_i=(V_i,E_i)$, with $V_i=\{T_i,F_i,a_i^1,a_i^2,b_i^1,b_i^2\}$ and $E_i=\{T_ia_i^1,T_ia_i^2,b_i^1a_i^1,b_i^2a_i^2,F_ib_i^1,F_ib_i^2\}$ (see Figure \ref{figure:NP1} for reference).
The nodes $T_i$ and $F_i$ are the \texttt{TRUE} and \texttt{FALSE} ends of the component, respectively. Each component is connected with the rest of the graph only through these two nodes which give us the following claim:

\begin{claim}\label{remark:NP1}
Let $u_i$ be an arbitrary variable in $U$. Any edge metric generator must contain at least one of the vertices $\{a_i^1,a_i^2,b_i^1,b_i^2\}$.
\end{claim}

\begin{proof}
Suppose that there exists an edge metric generator $S$ without any of these vertices in it. Since the component $X_i$ is attached to the rest of the graph only through the vertices $T_i$ and $F_i$, due to the symmetry, this implies that the edges $T_ia_i^1$ and $T_ia_i^2$ have the same distances to all vertices in the set $S$, a contradiction.
\end{proof}

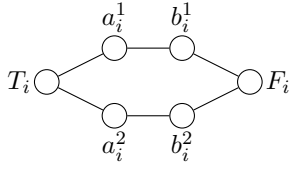
\begin{figure}[h]
\centering
\begin{tikzpicture}[scale=.9, transform shape]

\node [draw, shape=circle] (a1) at (-6,6) {};
\node [draw, shape=circle] (a2) at (-5,5.5) {};
\node [draw, shape=circle] (a3) at (-4,5.5) {};
\node [draw, shape=circle] (a4) at (-3,6) {};
\node [draw, shape=circle] (a5) at (-4,6.5) {};
\node [draw, shape=circle] (a6) at (-5,6.5) {};

\draw (-6.1,6) node[left] {$T_i$};
\draw (-2.9,6) node[right] {$F_i$};
\draw (-4,5.4) node[below] {$b_i^2$};
\draw (-5,5.4) node[below] {$a_i^2$};
\draw (-4,6.6) node[above] {$b_i^1$};
\draw (-5,6.6) node[above] {$a_i^1$};

\foreach \from/\to in {
 a1/a2, a2/a3, a3/a4, a4/a5, a5/a6, a6/a1}
\draw (\from) -- (\to);
\end{tikzpicture}
\caption{The truth-setting component for variable $u_i$.}\label{figure:NP1}
\end{figure}

Now, suppose that $c_j=y_j^1 \vee y_j^2 \vee y_j^3$, where $y_j^k$ is a literal in the clause $c_j$. For such clause $c_j$, we construct a satisfaction testing component $Y_j=(V_j',E_j')$, with $V_j'=\{c_j^1,\ldots, c_j^{10}\}$ and $E_j'=\{c_j^1c_j^2,c_j^1c_j^3,c_j^4c_j^2,c_j^4c_j^3,c_j^2c_j^5,c_j^5c_j^6,c_j^5c_j^7,c_j^3c_j^8,c_j^8c_j^9,c_j^8c_j^{10}\}$ (see Figure \ref{figure:NP2} for reference). The component is attached to the rest of the graph only through vertices $c_j^1$ and $c_j^2$ which give us the following claim.

\begin{claim}\label{remark:NP2}
Let $c_j$ be an arbitrary clause in $C$. Any edge metric generator must contain at least one of the vertices $\{c_j^6,c_j^7\}$ and at least one of the vertices $\{c_j^9,c_j^{10}\}$.
\end{claim}

\begin{proof}
Suppose that there exists an edge metric generator $S$ without any of the vertices $\{c_j^6,c_j^7\}$ in it. Since all the shortest paths from any vertex $x\ne c_j^6,c_j^7$ to the edges $c_j^5c_j^6$ and $c_j^5c_j^7$ go through the vertex $c_j^5$, this implies that the edges $c_j^5c_j^6, c_j^5c_j^7$ have the same distance to all vertices in the set $S$, a contradiction. A similar process works for the vertices $\{c_j^9,c_j^{10}\}$.
\end{proof}

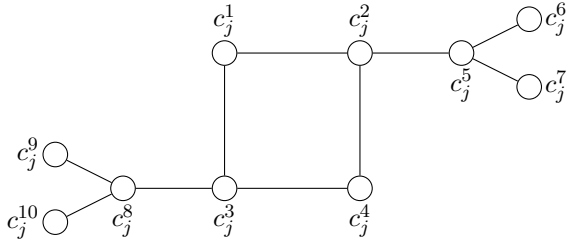
\begin{figure}[h]
\centering
\begin{tikzpicture}[scale=.9, transform shape]

\node [draw, shape=circle] (a1) at (-1,0) {};
\node [draw, shape=circle] (a2) at (1,0) {};
\node [draw, shape=circle] (a3) at (-1,-2) {};
\node [draw, shape=circle] (a4) at (1,-2) {};
\node [draw, shape=circle] (a5) at (2.5,0) {};
\node [draw, shape=circle] (a6) at (3.5,0.5) {};
\node [draw, shape=circle] (a7) at (3.5,-0.5) {};
\node [draw, shape=circle] (a8) at (-2.5,-2) {};
\node [draw, shape=circle] (a9) at (-3.5,-1.5) {};
\node [draw, shape=circle] (a10) at (-3.5,-2.5) {};

\draw (-1,0.1) node[above] {$c_j^1$};
\draw (1,0.1) node[above] {$c_j^2$};
\draw (-1,-2.1) node[below] {$c_j^3$};
\draw (1,-2.1) node[below] {$c_j^4$};
\draw (2.5,-0.1) node[below] {$c_j^5$};
\draw (3.6,0.5) node[right] {$c_j^6$};
\draw (3.6,-0.5) node[right] {$c_j^7$};
\draw (-2.5,-2.1) node[below] {$c_j^8$};
\draw (-3.6,-1.5) node[left] {$c_j^9$};
\draw (-3.6,-2.5) node[left] {$c_j^{10}$};


\foreach \from/\to in {
 a1/a2, a1/a3, a2/a4, a2/a5, a3/a4, a3/a8, a5/a6, a5/a7, a8/a9, a8/a10}
\draw (\from) -- (\to);
\end{tikzpicture}
\caption{The satisfaction testing component for clause $c_j$.}\label{figure:NP2}
\end{figure}

We also add some edges between truth-setting and satisfaction testing components as follows. If a variable $u_i$ occurs as a positive literal in a clause $c_j$, then we add the edges $T_ic_j^1$ and $F_ic_j^2$. If a variable $u_i$ occurs as a negative literal in a clause $c_j$, then we add the edges $T_ic_j^2$ and $F_ic_j^1$. For each clause $c_j \in C$ denote those six added edges with $E_j''$. We call them \textit{communication} edges. Figure \ref{figure:NP3} shows the edges that were added corresponding to the clause $c_j= (u_1 \vee \overline{u_2} \vee u_3)$, where $\overline{u_2}$ represents the negative literal corresponding to the variable $u_2$.

For all $k \in \{1, \ldots, n\}$ such that neither of $u_k$ and $\overline{u_k}$ occur in clause $c_j$, add the edges $T_kc_j^2$ to the graph $G$. For each clause $c_j \in C$ denote them with $E_j'''$. Those edges keep the graph to be connected. We call them \textit{neutralizing} edges, because no matter what value is assigned to the variable $u_k$ (or equivalently which vertex $v_k$ from the corresponding truth-setting component is chosen for an edge metric generator), this gives the same distance from such $v_k$ to the edges $c_j^1c_j^2$ and $c_j^2c_j^4$ from the satisfaction testing component corresponding to the clause $c_j$. These two edges play an important role later in the proof.

Finally, for each clause $c_j$ and every $k \in \{1, \ldots, m\}, k \neq j$, add the edges $c_j^2c_k^2$ to the graph $G$. For each clause $c_j \in C$ denote them with $E_j''''$. We call these edges as \textit{correcting} edges.

\begin{figure}[h]
\centering
\begin{tikzpicture}[scale=.9, transform shape]

\node [draw, shape=circle] (a1) at (-1,2) {};
\node [draw, shape=circle] (a2) at (1,2) {};
\node [draw, shape=circle] (a3) at (-1,0) {};
\node [draw, shape=circle] (a4) at (1,0) {};
\node [draw, shape=circle] (a5) at (2.5,2) {};
\node [draw, shape=circle] (a6) at (3.5,2.5) {};
\node [draw, shape=circle] (a7) at (3.5,1.5) {};
\node [draw, shape=circle] (a8) at (-2.5,0) {};
\node [draw, shape=circle] (a9) at (-3.5,0.5) {};
\node [draw, shape=circle] (a10) at (-3.5,-0.5) {};

\draw (-1,1.65) node[left] {$c_j^1$};
\draw (1,1.65) node[right] {$c_j^2$};


\node [draw, shape=circle] (a11) at (-6,6) {};
\node [draw, shape=circle] (a12) at (-5,5.5) {};
\node [draw, shape=circle] (a13) at (-4,5.5) {};
\node [draw, shape=circle] (a14) at (-3,6) {};
\node [draw, shape=circle] (a15) at (-4,6.5) {};
\node [draw, shape=circle] (a16) at (-5,6.5) {};

\draw (-6.1,6) node[left] {$T_1$};
\draw (-2.9,6) node[right] {$F_1$};

\node [draw, shape=circle] (a17) at (3,6) {};
\node [draw, shape=circle] (a18) at (4,5.5) {};
\node [draw, shape=circle] (a19) at (5,5.5) {};
\node [draw, shape=circle] (a20) at (6,6) {};
\node [draw, shape=circle] (a21) at (5,6.5) {};
\node [draw, shape=circle] (a22) at (4,6.5) {};

\draw (2.9,6) node[left] {$T_3$};
\draw (6.1,6) node[right] {$F_3$};

\node [draw, shape=circle] (a23) at (-1.5,7) {};
\node [draw, shape=circle] (a24) at (-0.5,6.5) {};
\node [draw, shape=circle] (a25) at (0.5,6.5) {};
\node [draw, shape=circle] (a26) at (1.5,7) {};
\node [draw, shape=circle] (a27) at (0.5,7.5) {};
\node [draw, shape=circle] (a28) at (-0.5,7.5) {};

\draw (-1.6,7) node[left] {$T_2$};
\draw (1.6,7) node[right] {$F_2$};

\foreach \from/\to in {
 a1/a2, a1/a3, a2/a4, a2/a5, a3/a4, a3/a8, a5/a6, a5/a7, a8/a9, a8/a10,
 a11/a12, a12/a13, a13/a14, a14/a15, a15/a16, a16/a11,
 a17/a18, a18/a19, a19/a20, a20/a21, a21/a22, a22/a17,
 a23/a24, a24/a25, a25/a26, a26/a27, a27/a28, a28/a23,
 a11/a1, a14/a2, a17/a1, a20/a2, a23/a2, a26/a1}
\draw (\from) -- (\to);
\end{tikzpicture}
\caption{The subgraph associated to the clause $c_j=(u_1\vee \overline{u_2} \vee u_3)$.}\label{figure:NP3}
\end{figure}
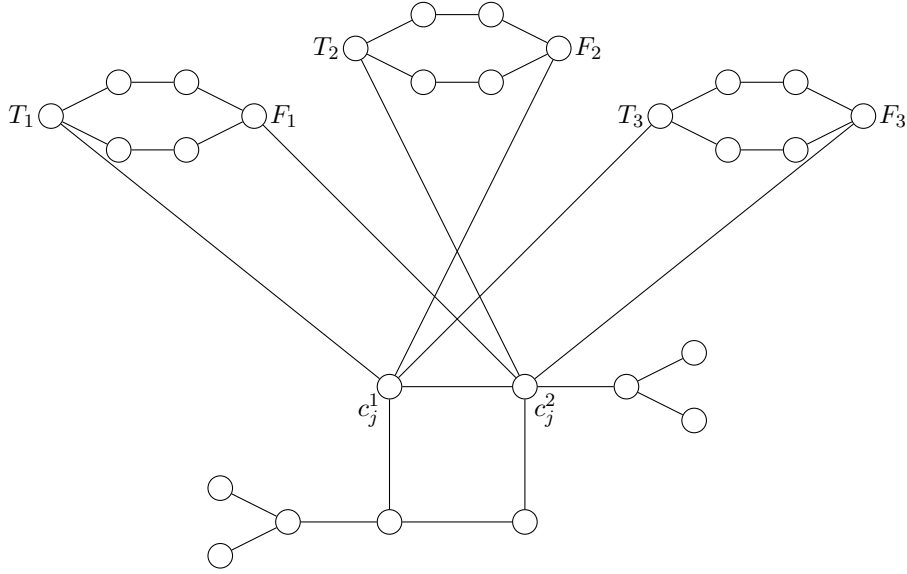

The construction of our instance of the EDIM problem is then completed by setting $r=2m+n$ and $G=(V,E)$, where
$$ V= \left( \bigcup_{i=1}^n V_i \right) \cup \left( \bigcup_{j=1}^m V_j' \right)$$
and
$$ E= \left( \bigcup_{i=1}^n E_i \right) \cup \left( \bigcup_{j=1}^m E_j' \right) \cup \left( \bigcup_{j=1}^m E_j'' \right) \cup \left( \bigcup_{j=1}^m E_j''' \right) \cup \left( \bigcup_{j=1}^m E_j'''' \right) $$

It is not hard too see that the construction can be done in polynomial time. It remains to show that $C$ is satisfiable if and only if $G$ has an edge metric generator of size $r$. From Claims \ref{remark:NP1} and \ref{remark:NP2} we get the following.

\begin{corollary}
The edge metric dimension of the graph $G$ is at least $r=2m+n$.
\end{corollary}

We now continue with the following lemmas which constitutes the heart of our NP-completeness reduction from 3-SAT.

\begin{lemma}\label{lema:NP1}
If $C$ is satisfiable, then the edge metric dimension of graph $G$ is $r$.
\end{lemma}

\begin{proof}
We know that the edge metric dimension is at least $r$. We now construct an edge metric generator $S$ of size $r$ based on a satisfying truth assignment for $C$.
Let $t:U \rightarrow \{\texttt{TRUE,FALSE}\}$ be a satisfying truth assignment for $C$. For each clause $c_j \in C$ put in the set $S$ vertices $c_j^6$ and $c_j^9$. For each variable $u_i \in U$ put in the set $S$ either the vertex $a_i^1$ if $t(u_i)=\texttt{TRUE}$, or the vertex $b_i^1$ if $t(u_i)=\texttt{FALSE}$.
We now show that $S$ is an edge metric generator for the graph $G$.

Let $e_{j,k}$ be an arbitrary correcting edge between the satisfaction testing components $c_j$ and $c_k$. We notice that $e_{j,k}$ is uniquely determined by the set of vertices $\{c_j^6,c_k^6\}$, because this is the only edge in the graph $G$ having distance 2 to both of them $c_j^6$ and $c_k^6$.

Let $i \in \{1, \ldots,n\}$ and $j \in \{1, \ldots,m\}$ be arbitrary indexes and let $v_i \in V_i \cap S$. Since we have already checked that any correcting edge is uniquely determined by some vertices in $S$, we do not have to check any pair of edges in which occur at least one correcting edge.
Also, it is easy to check that each communication edge and each neutralizing edge between a truth-setting component $X_i$ and a satisfaction testing component $Y_j$ is distinguished from all the remaining edges by the vertices $ v_i$, $c_j^6$ and $c_j^9$.

We next take a look at the edges in a truth-setting component. Let $i \in  \{1, \ldots,n\}$ be an arbitrary index and let $e \in E_i$ be an arbitrary edge from $X_i$. Since we have already checked that all correcting, communication and neutralizing edges are distinguished by some vertices from $S$ we only need to check that $e$ has different distance vectors: (1) from all other edges in $X_i$, (2) from all edges in other truth-setting components, and (3) from all edges in the satisfaction testing components. This is addressed at next. (1) For checking that $e$ has different distance vectors to all other edges in $X_i$, we consider two possibilities.
\begin{itemize}
\item $u_i$ or $\overline{u}_i$ is a literal in at least one clause $c_j$. Thus, the vertices $v_i$, $c_j^6$ and $c_j^9$ distinguish the edge $e$ from all other edges in $X_i$.
\item neither $u_i$ nor $\overline{u}_i$ are literals in any clause $c_j$. Thus, for an arbitrary $j \in \{1, \ldots, m\}$, the vertices $v_i$, $c_j^6$ distinguish the edge $e$ from all other edges in $X_i$.
\end{itemize}
For (2), let $k \in  \{1, \ldots,n\}, k\neq i$, be an arbitrary index. The vertex $v_i$ distinguishes the edge $e$ from all edges $f \in E_k$ (the edges in the truth-setting component $X_k$). For (3), let $j \in  \{1, \ldots,m\}$ be an arbitrary index. Hence, the vertices $c_j^6$ and $c_j^9$ distinguish edge $e$ from all edges $f \in E_j'$ (the edges in the satisfaction testing component $Y_j$).

Finally, we take a look at the edges from the satisfaction testing components.
Let $j \in  \{1, \ldots,m\}$ be an arbitrary index. Each one of the edges $\{c_j^2c_j^5,c_j^5c_j^6,c_j^5c_j^7,c_j^3c_j^8,c_j^8c_j^9,c_j^8c_j^{10}\}$ is uniquely determined by the set of vertices $\{c_j^6,c_j^9\}$. Those two vertices also distinguish the edges $c_j^1c_j^2, c_j^2c_j^4$ from all the other edges. Similarly, the same holds for the edges $c_j^1c_j^3, c_j^3c_j^4$.
To complete the proof, we need to show that for precisely this pair of edges there exists a vertex in the set $S$ that distinguish them. Since $C$ is satisfiable, suppose that $c_j$ is satisfied by the variable $u_i$. For the variable $u_i$ there are two possibilities:
\begin{itemize}
\item $u_i$ occurs as a positive literal in $c_j$ and $t(u_i)=\texttt{TRUE}$
\item $u_i$ occurs as a negative literal in $c_j$ and $t(u_i)=\texttt{FALSE}$.
\end{itemize}
Thus, if $t(u_i)=\texttt{TRUE}$, then we have added the vertex $a_i^1$ to the set $S$. In such case, the distance from $a_i^1$ to the edge $c_j^1c_j^2$ is 2, while the distance to the edge $c_j^2c_j^4$ is 3. Similarly, the distance from $a_i^1$ to the edge $c_j^1c_j^3$ is 2 and to the edge $c_j^3c_j^4$ is 3. The case when $t(u_i)=\texttt{FALSE}$ is symmetric.

Therefore, any two edges are distinguished by a vertex of $S$, and as as consequence, $S$ is an edge metric generator for graph $G$, which completes the proof of this lemma.
\end{proof}

\begin{lemma}\label{lema:NP2}
If the edge metric dimension of graph $G$ is $r$, then $C$ is satisfiable.
\end{lemma}

\begin{proof}
Let $S$ be an arbitrary edge metric generator for graph $G$ with cardinality $r$. From Claims \ref{remark:NP1} and \ref{remark:NP2}, the set $S$ must contain at least one vertex from each truth-setting component and at least two vertices from each satisfaction testing component. Since the cardinality of $S$ equals $r=2m+n$, it follows that in the set $S$ there is exactly one vertex from each truth-setting component and exactly two vertices from each satisfaction testing component.
We shall find a function $t: U \rightarrow \{\texttt{TRUE,FALSE}\}$ such that it represents a satisfying truth assignment for the collection of clauses $C$. For an arbitrary $i\in \{1, \ldots, n\}$, let $v_i \in V_i \cap S$. Hence, we define a function $t$ as follows:
$$t(u_i)=\left\{\begin{array}{ll}
                           \texttt{TRUE}, & v_i \in \{a_i^1,a_i^2\}, \\
                           \texttt{FALSE}, & v_i \in \{b_i^1,b_i^2\}.
                         \end{array}
\right.$$
We shall show that $t$ produces a satisfying truth assignment for $C$. To this end, let $c_j$ be an arbitrary clause. We claim that at least one of its literals has value \texttt{TRUE}. We prove that fact, by tracing which vertex from $S$ distinguishes the edges $e_j^1=c_j^1c_j^2$ and $e_j^2=c_j^2c_j^4$, and showing that the corresponding function $t$ satisfies $c_j$.

Let $k \in  \{1, \ldots,m\}$ be an arbitrary index. For the clause $c_k$ we assume, without loss of generality, that the vertices in the set $S$ are $c_k^6$ and $c_k^9$. If $j=k$, then both edges $e_j^1$ and $e_j^2$ are at distance 2 from $c_k^6$ and at distance 3 from $c_k^9$. If $j \neq k$, then by using the correcting edges, we deduce that the edges $e_j^1$ and $e_j^2$ are at distance 3 from $c_k^6$ and at distance 5 from $c_k^9$. Therefore, none of these vertices distinguish  $e_j^1$ from $e_j^2$.

Now, consider any variable $u_i$ which does not occur in $c_j$. If $v_i \in \{a_i^1,a_i^2\}$, then both edges $e_j^1, e_j^2$ are at distance 2 from $v_i$. If $v_i \in \{b_i^1,b_i^2\}$, then both edges are at distance 3 from $v_i$. Thus, the vertex of $S$ distinguishing the edges $e_j^1,e_j^2$ must belong to one of the truth-setting components that corresponds to a variable $u_k$ that occurs in the clause $c_j$. We recall that we have added communication edges in such a manner that $v_k$ distinguishes the edges $e_j^1$ and $e_j^2$ only if one of the following statements holds:
\begin{itemize}
\item $u_k$ occurs as a positive literal in $c_j$ and $v_k \in \{a_k^1,a_k^2\}$ - in this case $t(u_k)=\texttt{TRUE}$;
\item $u_k$ occurs as a negative literal in $c_j$ and $v_k \in \{b_k^1,b_k^2\}$ - in this case $t(u_k)=\texttt{FALSE}$;
\end{itemize}
In both cases the clause $c_j$ is satisfied by the setting assigned to the variable $u_k$. As a consequence, the formula $C$ is satisfiable, which completes the proof of this lemma.
\end{proof}

As a consequence of the Lemmas \ref{lema:NP1} and \ref{lema:NP2} above, the polynomial transformation from 3-SAT to the EDIM problem is done, and the proof of the theorem is now completed.
\end{proof}

As a consequence of Theorem \ref{theorem:NP-complete} we have the following result.

\begin{corollary}\label{np-hard}
The problem of finding the edge metric dimension of a connected graph is NP-hard.
\end{corollary}

\subsection{Approximation of the EDIM problem}

In concordance with Corollary \ref{np-hard}, finding the edge metric dimension of a graph is NP-hard in general. Thus, it is reasonable to look for an approximation algorithm for it. We use an approach similar to that in \cite{landmarks} getting an approximation in polynomial time within a factor of $O(\log{m})$ where $m$ is the number of edges of the graph. We show that the problem of finding the edge metric dimension can be transformed in polynomial time to the set cover problem. Once we have the set cover problem we use the $O(\log m)$ factor approximation algorithm for the set cover problem \cite{approximationAlgorithms} to obtain an approximation algorithm for the EDIM problem.

\begin{theorem}
Let $G=(V,E)$ be an arbitrary connected graph with $m$ edges. Then $edim(G)$ can be approximated within a factor of $O(\log m)$ in polynomial time.
\end{theorem}

\begin{proof}
Starting from a graph $G$ we first construct an instance of the set cover problem, similarly to the one in \cite{approximationAlgorithms}. Let $F$ be a finite family $\{S_1, S_2, \ldots, S_p\}$ of finite sets and let $U= \bigcup_{S \in F}S$ be the universe set. We look for a subfamily $F' \subseteq F$ with minimum cardinality for which it holds that $\bigcup_{S \in F'}S = U$.

For each vertex in the graph $G$ we can compute in polynomial time all the pairs of edges that have different distance to that vertex. For a vertex $v$, denote  with $S_v$ the set of all such pairs of edges. To solve the EDIM problem one has to find a set of vertices $S$ with minimum cardinality such that every pair of edges is distinguished by some vertex $v \in S$.
We can easily transform the EDIM problem to the set cover problem by setting $F=\{S_{v_1}, \ldots, S_{v_n}\}$, where $v_1, \ldots, v_n$ are all the vertices from the graph $G$. Observe that the universe set $U$ is the set of all possible pairs of edges in the graph $G$ with cardinality $\binom m2$.
It is not hard to see that there exists an edge metric basis of size $k$ if and only if there is a set cover of size $k$.

For the set cover problem there is a polynomial approximation algorithm that finds a set cover within a factor of $O(\log m)$. Therefore, we get the same approximation for the EDIM problem.
\end{proof}

\section{Some bounds and closed formulae}


It is clear that for any vertex $v$ of a connected $G$, the set $V(G)-\{v\}$ is an edge metric generator. Also, it is necessary at least to have one vertex in any edge metric generator. Thus, natural bounds on the edge metric dimension of a graph are the following ones. For any connected graph $G$ of order $n$,
\begin{equation}\label{general-bound}
1\le edim(G)\le n-1.
\end{equation}

The graphs achieving the equality in the lower bound above is relatively easy to deal with, being the same as for the standard metric dimension. This was already given in Remark \ref{rem-path-1}. However, for the upper bound, characterizing all the graphs satisfying the equality is not exactly clear how to settle, which is quite different from the standard metric dimension, where it is known that $dim(G)=n-1$ if and only if $G$ is a complete graph.



\begin{proposition}
Let $G$ be a connected graph of order $n$ and $edim(G)=n-1$. Then for every $u, v \in V(G)$, $u \neq v$ it holds $N(u) \cap N(v) \neq \emptyset$.
\end{proposition}

\begin{proof}
If there are two distinct vertices $u$ and $v$ such that $N(u) \cap N(v) = \emptyset$, then we will show that $S = V(G) \setminus \lbrace u,v \rbrace$ is an edge metric generator. Let $e$ be an edge of $G$. Then we have the following options:
\begin{itemize}
\item If $e=xy$, where $x,y \in S$, then $e$ has distance $0$ to exactly two vertices is $S$, i.e. $x$ and $y$.
\item If $e = xu$ or $e= xv$, where $x \in S$, then $e$ has distance $0$ just to one vertex in $S$ - this is $x$.
\item If $e = uv$, then $e$ has distance more than $0$ to every vertex in $S$.
\end{itemize}
It is obvious that two edges $e$ and $f$ can have the same distance to every vertex in $S$ only in the case when $e = xu$ and $f = xv$ for some vertex $x \in S$. But we assumed that $N(u) \cap N(v) = \emptyset$ and therefore, this case can not happen. Hence, $S$ is an edge resolving set and $edim(G) \leq n-2$, a contradiction.
\end{proof}

\begin{proposition}
Let $G$ be a connected graph of order $n$. If there is a vertex $v \in V(G) $ of degree $n-1$, then either $edim(G) = n-1$ or $edim(G)=n-2$.
\end{proposition}

\begin{proof}
Let $x$ and $y$ be distinct vertices, different from $v$ and $S \subseteq V(G) \setminus \lbrace x, y \rbrace$. If $e = xv$ and $f = yv$, then $d(e,v) = d(f,v)= 0$ and $d(e,z)=d(f,z)=1$ for every $z \in S \setminus \lbrace v\rbrace$. Therefore, $S$ can not be an edge resolving set. It follows that any edge resolving set contains all vertices of $G$, except maybe $v$ and one other vertex. Hence, $edim(G) \geq n-2$.
\end{proof}

\begin{proposition}
Let $G$ be a connected graph of order $n$. If there are two distinct vertices $u, v \in V(G) $ of degree $n-1$, then $edim(G) = n-1$.
\end{proposition}

\begin{proof}
We will show that every $S \subseteq V(G)$, which does not contain exactly two vertices of $G$, is not an edge metric generator. We consider two cases:
\begin{enumerate}
\item If $u$ and $v$ are not in $S$: let $e = xu$ and $f=xv$, where $x \in S$. Then $e$ and $f$ both have distance $0$ to $x$ and distance $1$ to every other vertex in $S$.
\item If at least one of the vertices $u$ and $v$ is in $S$: without loss of generality assume that $v \in S$ and $S = V(G) \setminus \lbrace x,y \rbrace$ where $x,y \in V(G) \setminus \lbrace v \rbrace$. Let $e = vy$ and $f = vx$. Then $e$ and $f$ both have distance $0$ to $v$ and distance $1$ to every other vertex in $S$.

\end{enumerate}
In both cases we can find two edges with the same distance to every vertex in $S$. Therefore, $S$ is not an edge metric generator. With this we have proved that $edim(G) = n-1$.
\end{proof}

We observe that there are graphs $G$ of order $n$ and maximum degree strictly less than $n-1$ for which $edim(G) = n-1$. The circulant graph\footnote{A circulant graph $CR(n,r)$ is a graph of order $n$ with vertex set $V=\{v_0,v_1,\dots,v_{n-1}\}$ such that $v_i$ is adjacent to $v_{i+j}$ with $j\in \{1,\dots,r\}$, $i\in \{0,\dots,n-1\}$ and the operation $i+j$ is done modulo $n$.} $CR(6,2)$ is a simple example of this, which leads to think that not only graphs $G$ of order $n$ and maximum degree $n-1$ satisfy that $edim(G) = n-1$.


We now continue with several bounds on the edge metric dimension of connected graphs. Some of these general bounds are obtained by using the approach of the edge metric representation of edges with respect to an edge metric basis.

\begin{proposition}
Let $G$ be connected graph and let $\Delta(G)$ be the maximum degree of $G$. Then,
$$edim(G) \geq \left\lceil \log_2{\Delta(G)} \right\rceil.$$
\end{proposition}

\begin{proof}
From an arbitrary vertex $v \in V(G)$ there can be only two different distances to some set of incident edges. Therefore, to distinguish all edges that for one endpoint has the vertex $u$ with $\deg{u}=\Delta(G)$, it must hold $2^{edim(G)} \geq \Delta(G)$ and the assertion follows.
\end{proof}

\begin{proposition}
Let $G$ be a connected graph and let $S$ be an edge metric basis with $|S|=k$. Then $S$ does not contain a vertex with degree greater than $2^{k-1}$.
\end{proposition}

\begin{proof}
Suppose that there exists an edge metric basis with a vertex $v$ of degree greater than $2^{k-1}$. The incident edges with endpoint $v$ have all equal distance to $v$. So, there remain $k-1$ vertices to distinguish all those incident edges. Since from an arbitrary vertex $u \in V(G)$ there can be only two different distances to the set of incident edges it follows that this is not an edge metric generator. We get a contradiction with our assumption, so all vertices in an edge metric basis are of degree smaller or equal to $2^{k-1}$.
\end{proof}

\begin{proposition}
Let $G$ be a connected graph. If $edim(G)=k$ and $G$ has diameter $D$, then $|E(G)| \leq (D+1)^k$.
\end{proposition}

\begin{proof}
Since the diameter of the graph $G$ equals $D$, the distance from an arbitrary vertex to an arbitrary edge in the graph $G$ can get values from $0$ to $D$. Therefore an edge metric basis can distinguish at most $(D+1)^k$ edges, and therefore the graph $G$ cannot have more edges.
\end{proof}


We next study the edge metric dimension of hypercubes graphs $Q_n$. To this end, we use a binary representation of $Q_n$. That is, the vertex set of $Q_n$ consists of the $2^n$-dimensional boolean vectors, \emph{i.e.}, vectors with binary coordinates $0$ or $1$, and two vertices are adjacent whenever they differ in exactly one coordinate. It is known (see \cite{erdos}) that in any $n$-dimensional hypercube, the set of vertices $B_n = \lbrace 11\ldots11, 01\ldots11, 10\ldots11, \ldots, 11\ldots01 \rbrace$ is a metric generator. We will prove that this set is also an edge metric generator for $Q_n$.

\begin{theorem}
Let $n$ be a positive integer and let $Q_n$ the $n$-dimensional hypercube. Then $edim(Q_n) \leq n$.
\end{theorem}

\begin{proof}
We will show that the set of $n$ vertices $B_n = \lbrace 11\ldots11, 01\ldots11, 10\ldots11, \ldots, 11\ldots01 \rbrace$ is an edge metric generator. If $n=1$, this result follows immediately. Therefore, we assume that $n > 1$.
Let $e=uv$ and $f=xy$ be two different edges in $Q_n$. It suffices to prove that there exist $z \in B_n$ such that $d(e,z) \neq d(f,z)$. Suppose that this is not true. Thus, for every $z \in B_n$ it holds $d(e,z) = d(f,z)$. Of course, there is exactly one coordinate, let say $i$, such that $u_i \neq v_i$ and there is exactly one coordinate, let say $j$, such that $x_j \neq y_j$. Consider the following two cases.
\begin{enumerate}
\item $i \neq j$, and without loss of generality, let $i < j$: \\
Let $E$ be the number of coordinates $k \in \lbrace 1,2,\ldots,n \rbrace \setminus \lbrace i, j \rbrace $, such that $u_k=v_k = 0$. Furthermore, let $F$ be the number of coordinates $k \in \lbrace 1,2,\ldots,n \rbrace \setminus \lbrace i, j \rbrace $, such that $x_k=y_k = 0$.
\begin{itemize}
\item If $x_i=y_i=u_j=v_j = 0$ or $x_i=y_i=u_j=v_j = 1$, then since $d(e, 11\ldots11)=d(f,11\ldots11)$, it follows that $E=F$. Let $z \in B_n$ be a vertex with $z_i = 0$. Thus, ($d(e,z) = E +1$ and $d(f,z) = F$) or ($d(e,z) = E$ and $d(f,z) = F+1$). Therefore, $d(e,z) \neq d(f,z)$, a contradiction.

\item If ($x_i=y_i = 0$ and $u_j = v_j =1$) or ($x_i=y_i = 1$ and $u_j = v_j = 0$), then let $z \in B_n$ be a vertex with $z_i = 0$. Since $d(e,z)=d(f,z)$, it follows that $E=F$. Thus, ($d(e,11\ldots11) = E +1$ and $d(f,11\ldots11) = F$) or ($d(e,11\ldots11) = E$ and $d(f,11\ldots11) = F +1$). Therefore, $d(e,11\ldots11) \neq d(f,11\ldots11)$, a contradiction.
\end{itemize}
\item $i = j$: \\
In this case, let $B_{n-1}$ be a metric generator for the hypercube $Q_{n-1}$ as proved in \cite{erdos}. Let $E$ be a vertex in $Q_{n-1}$, obtained by deleting $i$-th coordinate in the vertex $u$, and let $F$ be a vertex in $Q_{n-1}$, obtained by deleting $i$-th coordinate in the vertex $x$. Since the edges $e$ and $f$ are different, it follows $E\neq F$. Also, for every $w \in B_{n-1}$ there is some $z_w \in B_n$, such that $w$ is obtained from $z_w$ by deleting the $i$-th coordinate. Since $d(E,w) = d(e,z_w) = d(f,z_w) = d(F,w)$ for every $w \in B_{n-1}$, we have $d(E,w) = d(F,w)$ for every $w \in B_{n-1}$. Since $B_{n-1}$ is a metric generator in $Q_{n-1}$, this is a contradiction.
\end{enumerate}
We have proved that for every two distinct edges $e$ and $f$ in the hypercube $Q_n$, it holds that there is $z \in B_n$ such that $d(e,z) \neq d(f,z)$. Therefore, $B_n$ is an edge metric generator and the bound is obtained.
\end{proof}

\section{Conclusion}

In this article we have introduced and initiated the study of a new variant of metric dimension in connected graphs concerning uniquely identifying the edges of the graph, namely the edge metric dimension. We have given some realization results on this new parameter in connection with the standard metric dimension and also, some comparison between both mentioned parameters. In addition, we have proved that computing the edge metric dimension of connected graphs is NP-hard throughout a polynomial reduction from the 3-SAT problem. We have also computed the value of the edge metric dimension of several graph families or bounded its value in some other cases. As the consequence of the study, several
questions which are of interest in order to continue the research in this direction could be posted. These are the following ones.
\begin{itemize}
\item Is it possible to completely settle the realization result concerning the triplet $r,t,n$ already mentioned in Subsection \ref{subsect-realiz}?
\item Is there a bound of $dim(G)$ in terms of $edim(G)$ or viceversa?
\item Can you characterize the families of graphs $G$ achieving the equality $edim(G)=dim(G)$?
\item Are there any other families of graph (different from the torus graph $C_{4r}\Box C_{4t}$) such that $edim(G)<dim(G)$?
\item The problem of computing the standard metric dimension of graph is proved to be NP-hard when restricted to planar graphs and it turns out polynomial for the case of outerplanar graphs (see \cite{diaz}). In this sense: Is it also true some similar result for the case of edge metric dimension?
\item Can you characterize the family of graphs $G$ of order $n$ satisfying that $edim(G)=n-1$?
\end{itemize}

\end{document}